\documentclass[aos, reqno, 11pt]{article}%

\usepackage{pdfsync}

\RequirePackage{amsthm, amsmath, amsfonts, amssymb}%

\usepackage{graphicx, color}%
\usepackage{mathrsfs}

\definecolor{darkblue}{rgb}{0.0,0.0,0.7}

\RequirePackage[%
colorlinks = true,%
linkcolor = darkblue,%
citecolor = darkblue,%
urlcolor = darkblue, %
]{hyperref}%

\hypersetup{%
  pdfauthor = {},%
  pdftitle = {Sharp oracle inequalities for the prediction of a
    high-dimensional matrix},%
  pdfcreator = {pdflatex},%
  pdfproducer = {pdflatex}}

\usepackage{amsmath}%
\usepackage{amssymb}%
\usepackage{amscd}%
\newcommand{\R}{\mathbb{R}}%
\newcommand{\inr}[1]{\langle #1 \rangle} 

\newcommand{\N}{\mathbb{N}}
\newcommand{\E}{\mathbb{E}}
\renewcommand{\P}{\mathbb{P}}

\newcommand{\cC}{{\cal C}}
\newcommand{\cL}{{\cal L}}

\newcommand{\cX}{{\cal X}}
\newcommand{\cF}{{\cal F}}

\newcommand{\cZ}{{\cal Z}}
\newcommand{\cM}{{\cal M}}

\newcommand{\var}{\text{Var}}%
\newcommand{\prodsca}[2]{\langle #1,#2 \rangle}%
\newcommand{\norm}[1]{\|#1\|}%

\DeclareMathOperator*{\Span}{span}

\DeclareMathOperator*{\argmin}{argmin}

\DeclareMathOperator*{\Star}{star}
\DeclareMathOperator*{\tr}{tr}
\DeclareMathOperator*{\ra}{rank}
\DeclareMathOperator{\ve}{vec}

\DeclareMathOperator{\pen}{pen}

\DeclareMathOperator{\rank}{rank}

\newcommand{\1}{{\rm 1}\kern-0.24em{\rm I}}

\newtheorem{theorem}{Theorem}%
\newtheorem{lemma}[theorem]{Lemma}%
\newtheorem{definition}[theorem]{Definition}%
\newtheorem{proposition}{Proposition}%
\newtheorem{corollary}{Corollary}%
\newtheorem{remark}{Remark}
\newtheorem{assumption}{Assumption}%

\begin{document}

\title{Sharp oracle inequalities for high-dimensional matrix
  prediction}

\date{\today}

\author{St\'ephane Ga\"iffas$^{1, 3}$ and Guillaume Lecu\'e$^{2, 3}$}

\footnotetext[1]{Universit\'e Pierre et Marie Curie - Paris~6,
  Laboratoire de Statistique Th\'eorique et Appliqu\'ee. \emph{email}:
  \texttt{stephane.gaiffas@upmc.fr}}

\footnotetext[2] {CNRS, Laboratoire d'Analyse et Math\'ematiques
  appliqu\'ees, Universit\'e Paris-Est - Marne-la-vall\'ee
  \emph{email}: \texttt{guillaume.lecue@univ-mlv.fr}}

\footnotetext[3]{This work is supported by French Agence Nationale de
  la Recherce (ANR) ANR Grant \textsc{``Prognostic''}
  ANR-09-JCJC-0101-01. (\texttt{http://www.lsta.upmc.fr/prognostic/index.php})
}

\maketitle

\begin{abstract}
  We observe $(X_i,Y_i)_{i=1}^n$ where the $Y_i$'s are real valued
  outputs and the $X_i$'s are $m\times T$ matrices. We observe a new
  entry $X$ and we want to predict the output $Y$ associated with it.
  We focus on the high-dimensional setting, where $m T \gg n$. This
  includes the matrix completion problem with noise, as well as other
  problems. We consider linear prediction procedures based on
  different penalizations, involving a mixture of several norms: the
  nuclear norm, the Frobenius norm and the $\ell_1$-norm. For these
  procedures, we prove sharp oracle inequalities, using a statistical
  learning theory point of view. A surprising fact in our results is
  that the rates of convergence do not depend on $m$ and $T$ directly.
  The analysis is conducted without the usually considered incoherency
  condition on the unknown matrix or restricted isometry condition on
  the sampling operator.  Moreover, our results are the first to give
  for this problem an analysis of penalization (such nuclear norm
  penalization) as a regularization algorithm: our oracle inequalities
  prove that these procedures have a prediction accuracy close to the
  deterministic oracle one, given that the reguralization
  parameters are well-chosen. \\

  \noindent%
\end{abstract}

\section{Introduction}
\label{sec:introduction}

\subsection{The model and some basic definitions}
\label{sec:model}

Let $(X, Y)$ and $D_n = (X_i, Y_i)_{i=1}^n$  be $n+1$ i.i.d
random variables with values in $ \cM_{m, T}\times \R$, where
$\cM_{m, T}$ is the set of matrices with $m$ rows and $T$ columns with
entries in $\R$. Based on the observations  $D_n$, we have in mind to predict the real-valued ouput $Y$ by a linear transform of the input variable $X$. We focus on the
high-dimensional setting, where $mT \gg n$. We use a ``statistical
learning theory point of view'': we do not assume that $\E(Y | X)$ has
a particular structure, such as $\E(Y | X) = \inr{X, A_0}$ for some
$A_0 \in \cM_{m, T}$, where $\inr{\cdot, \cdot}$ is the standard
Euclidean inner product given for any $ A, B \in \cM_{m, T}$ by
\begin{equation}
  \label{eq:inner-product}
  \prodsca{A}{B} := \tr(A^\top B).
\end{equation}
The statistical performance of a linear predictor $\inr{X,A}$ for some $A \in \cM_{m, T}$ is measured by the quadratic risk
\begin{equation}
  \label{eq:Quadratic-Risk}
  R(A) := \E [ ( Y - \inr{X,A} )^2 ].
\end{equation}
If $\hat A_n \in \cM_{m, T}$ is a statistic constructed from the observations $D_n$, then its risk is given
by the conditional expectation
\begin{equation*}
  R(\hat A_n) := \E [ ( Y - \inr{X, \hat A_n} )^2 | D_n].
\end{equation*}
A natural candidate for the prediction of $Y$ using $D_n$ is the
\textit{empirical risk minimization procedure}, namely any element in $\cM_{m,T}$ minimizing the empirical risk $R_n(\cdot)$ defined for all $A\in\cM_{m,T}$ by 
\begin{equation*}
   R_n(A) =
  \frac{1}{n} \sum_{i=1}^n (Y_i - \inr{X_i, A})^2.
\end{equation*}
It is well-known that the excess risk of this procedure is of order
$mT / n$. In the high dimensional setting, this rate is not going to
zero. So, if $X \mapsto \inr{A_0,X}$ is the best linear prediction of
$Y$ by $X$, we need to know more about $A_0$ in order to construct
algorithms with a small risk. In particular, we need to know that
$A_0$ has a ``low-dimensional structure''. In this setup, this is
usually done by assuming that $A_0$ is low rank. A first idea is then
to minimize $R_n$ and to penalize matrices with a large rank. Namely,
one can consider
\begin{equation}
  \label{eq:ERm-rank}
  \hat A_n \in \argmin_{A \in \cM_{m, T}} \big\{ R_n(A) + \lambda
  \rank(A) \big\},
\end{equation}
for some regularization parameter $\lambda > 0$. But $A \mapsto {\rm rank}(A)$ is far from being a
convex function, thus minimizing~\eqref{eq:ERm-rank} is very difficult
in practice, see~\cite{fazel2003log} for instance on this
problem. Convex relaxation of~\eqref{eq:ERm-rank} leads to the
following convex minimization problem
\begin{equation}
  \label{eq:ERm-S1}
  \hat A_n \in \argmin_{A \in \cM_{m, T}} \big\{ R_n(A) + \lambda
  \norm{A}_{S_1} \big\},
\end{equation}
where $\norm{\cdot}_{S_1}$ is the \emph{$1$-Schatten} norm, also known
as \emph{nuclear norm} or \emph{trace norm}. This comes from the fact
that the nuclear norm is the convex envelope of the rank on the unit
ball of the spectral norm, see~\cite{Fazel04rankminimization}.  For
any matrix $A \in \cM_{m, T}$, we denote by $s_1(A), \ldots,
s_{\ra(A)}(A)$ its nonincreasing sequence of singular values. For
every $p \in [1,\infty]$, the $p$-Schatten norm of $A$ is given by
\begin{equation}
  \label{eq:Schatten-norm}
  \norm{A}_{S_p} := \Big(\sum_{j=1}^{\ra(A)} s_j(A)^p \Big)^{1/p}.
\end{equation}
In particular, the $\norm{\cdot}_{S_\infty}$-norm is the
\emph{operator norm} or \emph{spectral norm}. The
$\norm{\cdot}_{S_2}$-norm is the \emph{Frobenius norm}, which
satisfies
\begin{equation*}
  \norm{A}_{S_2}^2 = \sum_{i, j} A_{i, j}^2 = \prodsca{A}{A}.
\end{equation*}

\subsection{Motivations}
\label{sec:motivations}

A particular case of the matrix prediction problem described in
Section~\ref{sec:model} is the problem of \emph{(noisy) matrix
  completion}, see \cite{srebro2005maximum, srebro2005rank}, which
became very popular because of the buzz surrounding the Netflix
prize\footnote{\texttt{http://www.netflixprize.com/}}. In this
problem, it is assumed that $X$ is uniformly distributed over the set
$\{ e_{p, q} : 1 \leq p \leq m, 1 \leq q \leq T \}$, where $e_{p,q}
\in \cM_{m, T}$ is such that $(e_{p,q})_{i,j} = 0$ when $i \neq q$ or
$j \neq p$ and $(e_{p,q})_{p,q} = 1$. If $\E(Y | X) = \inr{A_0, X}$
for some $A_0 \in \cM_{m, T}$, then the $Y_i$ are $n$ noisy
observations of the entries of $A_0$, and the aim is to denoise the
observed entries and to fill the non-observed ones. 

\emph{First motivation.} Quite surprisingly, for matrix completion
without noise ($Y_i = \prodsca{X_i}{A_0}$), it is proved
in~\cite{candes-recht08} and~\cite{candes-tao1} (see also
\cite{gross2009recovering}, \cite{recht2009simpler}) that nuclear norm
minimization is able, with a large probability (of order $1 - m^{-3})$
to recover \emph{exactly} $A_0$ when $n > c r (m + T) (\log n)^6$,
where $r$ is the rank of $A_0$. This result is proved under a
so-called \emph{incoherency} assumption on $A_0$. This assumption
requires, roughly, that the left and right singular vectors of $A_0$
are well-spread on the unit sphere. Using this incoherency assumption
\cite{candes-plan1}, \cite{DBLP:journals/corr/abs-0906-2027} give
results concerning the problem of matrix completion with
noise. However, recalling that this assumption was introduced in order
to prove \emph{exact} completion, and since in the noisy case it is
obvious that exact completion is impossible, a natural goal is then to
obtain results for noisy matrix completion without the incoherency
assumption. This is a first motivation of this work: we derive very
general sharp oracle inequalities without any assumption on $A_0$, not
even that it is low-rank. More than that, we don't need to assume that
$\E(Y | X) = \inr{X, A_0}$ for some $A_0$, since we use a statistical
learning point-of-view in the statement of our results. More
precisely, we construct procedures $\hat A_n$ satisfying \textit{sharp
  oracle inequalities} of the form
\begin{equation}
  \label{eq:oracle}
  R( \hat{A}_n) \leq \inf_{A \in \cM_{m, T}} \big\{ R(A) + r_{n}(A)
  \big\}
\end{equation}
that hold with a large probability, where $r_{n}(A)$ is a
\textit{residue} related to the penalty used in the definition of
$\hat A_n$ that we want as small as possible. By ``sharp'' we mean
that in the right hand side of~\eqref{eq:oracle}, the constant in
front of $R(A)$ is equal to one.

\emph{A surprising fact} in our results is that, for penalization
procedures that involve the 1-Schatten norm (and 2-Schatten norm if a
mixed penalization is considered), the residue $r_n(\cdot)$ does not depend
on $m$ and $T$ directly: it only depends on the 1-Schatten norm of
$A_0$, see Section~\ref{sec:main_results} for details. This was not,
as far as we know, previously noticed in literature (all the upper
bounds obtained for $\norm{\hat A_n - A_0}_{S_2}^2$ depend directly on
$m$ and $T$ and on $\norm{A_0}_{S_1}$ or on its rank and on
$\norm{A_0}_{S_\infty}$, see the references above and below). This
fact can be used to argue that $\norm{\cdot}_{S_1}$ is a better
measure of sparsity than the rank, and it points out an interesting
difference between nuclear-norm penalization (also called ``Matrix
Lasso'') and the Lasso for vectors.

In~\cite{rohde-tsyb09}, which is a work close to ours, upper bounds
for $p$-Schatten penalization procedures for $0 < p \leq 1$ are given
in the same setting as ours, including in particular the matrix
completion problem. The results are stated without the incoherency
assumption for matrix completion. But for this problem, the upper
bounds are given using the empirical norm $\norm{\hat A_n - A_0}_n^2 =
\sum_{i=1}^n \inr{X_i, \hat A_n - A_0}^2 / n$ only. An upper bound for
this measure of accuracy gives information only about the denoising
part and not about the filling part of the matrix completion
problem. Our results have the form~\eqref{eq:oracle}, and taking $A_0$
instead of the minimum in this equation gives an upper bound for
$R(\hat A_n) - R(A_0)$, which is equal to $\norm{\hat A_n -
  A_0}_{S_2}^2 / (m T)$ in the matrix completion problem when $\E(Y |
X) = \inr{X, A_0}$ (see Section~\ref{sec:main_results}).

\emph{Second motivation.} In the setting considered here, an
assumption called \emph{Restricted Isometry} (RI) on the sampling
operator $\cL(A) = (\inr{X_1, A}, \ldots, \inr{X_n, A}) / \sqrt n$ has
been introduced in \cite{recht2007guaranteed} and used in a series of
papers, see~\cite{rohde-tsyb09}, \cite{candes-plan2},
\cite{negahban2009unified, negahban2009estimation}. This assumption is
the matrix version of the restricted isometry assumption for vectors
introduced in \cite{candes2005decoding}. Note that in the
high-dimensional setting ($mT \gg n$), this assumption is not
satisfied in the matrix completion problem, see~\cite{rohde-tsyb09}
for instance, which works with and without this assumption. The RI
assumption is very restrictive and (up to now) is only  satisfied by some special random matrices (cf. \cite{MR2417886,MR2321621,MR2453368,MR2373017} and references therein). This is a second motivation for this work: our results do not
require any RI assumption. Our assumptions on $X$ are very mild, see
Section~\ref{sec:main_results}, and are satisfied in the matrix
completion problem, as well as other problems, such as the multi-task
learning.

\emph{Third motivation.} Our results are the first to give an analysis
of nuclear-norm penalization (and of other penalizations as well, see
below) as a regularization algorithm. Indeed, an oracle inequality of
the form~\eqref{eq:oracle} proves that these penalization procedures
have a prediction accuracy close to the deterministic oracle one,
given that the reguralization parameters are well-chosen.

\emph{Fourth motivation.} We give oracle inequalities for penalization
procedures involving a mixture of several norms: $\norm{\cdot}_{S_1}$,
$\norm{\cdot}_{S_2}^2$ and the $\ell_1$-norm $\norm{\cdot}_1$. As far
as we know, no result for penalization using several norms was
previously given in literature for high-dimensional matrix prediction.

Procedures based on 1-Schatten norm penalization have been considered
by many authors recently, with applications to multi-task learning and
collaborative filtering. The first studies are probably the ones given
in~\cite{srebro2005maximum,srebro2005rank}, using the hinge loss for
binary classification. In \cite{MR2417263}, it is proved, under some
condition on the $X_i$, that nuclear norm penalization can
consistently recover $\rank(A_0)$ when $n \rightarrow +\infty$. Let us
recall also the references we mentioned above and close other ones
\cite{Fazel04rankminimization,recht2007guaranteed},
\cite{candes-plan2,cai_candes_shen08,candes-recht08,candes-plan1,candes-tao1},
\cite{keshavan2009matrix, DBLP:journals/corr/abs-0906-2027},
\cite{rohde-tsyb09}, \cite{gross2009recovering},
\cite{recht2009simpler,recht2007guaranteed},
\cite{negahban2009unified,negahban2009estimation},
\cite{argyriou2010spectral,argyriou2008convex,argyriou2008spectral},
\cite{abernethy2009new}.


\subsection{The procedures studied in this work}

If $\E(Y | X) = \inr{X, A_0}$ where $A_0$ is low rank, in the sense
that $r \ll n$, nuclear norm penalization~\eqref{eq:ERm-S1} is likely
to enjoy some good prediction performances. But, if we know more about
the properties of $A_0$, then other penalization procedure can be
considered. For instance, if we know that the non-zero singular values
of $A_0$ are ``well-spread'' (that is almost equal) then it may be
interesting to use the ``regularization effect'' of a ``$S_2$ norm'' based  penalty in the same spirit as ``ridge type'' penalty for vectors or functions. Moreover, if we know that many entries of $A_0$ are close
or equal to zero, then using also a $\ell_1$-penalization
\begin{equation}
  \label{eq:penalty-sparse-metrix}
  A \mapsto \norm{A}_1 = \sum_{\substack{1\leq p\leq m\\1\leq q\leq
      T}} |A_{p,q}|
\end{equation}
may improve even further the prediction. In this paper, we
consider a penalization that uses a mixture of several norms: for
$\lambda_1, \lambda_2, \lambda_3 > 0$, we consider
\begin{equation}
  \label{eq:penalty-mixture}
  \pen_{\lambda_1,\lambda_2,\lambda_3}(A) =
  \lambda_1\norm{A}_{S_1} + \lambda_2\norm{A}_{S_2}^2 + \lambda_3\norm{A}_1
\end{equation}
and we will study the prediction properties of
\begin{equation}
  \label{eq:procedure}
  \hat A_n(\lambda_1,\lambda_2,\lambda_3) \in \argmin_{A \in
    \cM_{m,T}}\ \Big\{ R_n(A) + 
  \pen_{\lambda_1, \lambda_2, \lambda_3}(A) \Big\}.
\end{equation}
Of course, if more is known on the structure of $A_0$, other penalty functions can be considered.

We obtain sharp oracle inequalities for the procedure $\hat
A_n(\lambda_1,\lambda_2,\lambda_3)$ for any values of
$\lambda_1,\lambda_2,\lambda_3\geq 0$ (excepted for 
$(\lambda_1,\lambda_2,\lambda_3)=(0,0,0)$ which provides the well-studied empirical risk minimization procedure). In particular, depending on
the ``a priori'' knowledge that we have on $A_0$ we will consider
different values for the triple $(\lambda_1,\lambda_2,\lambda_3)$. If
$A_0$ is only known to be low-rank, one should choose $\lambda_1>0$
and $\lambda_2 = \lambda_3=0$. If $A_0$ is known to be low-rank with
many zero entries, one should choose $\lambda_1,\lambda_3 > 0$ and
$\lambda_2=0$. If $A_0$ is known to be low-rank with well-spread
non-zero singular values, one should choose $\lambda_1,\lambda_2 > 0$
and $\lambda_3 = 0$. Finally, one should choose $\lambda_1, \lambda_2,
\lambda_3 > 0$ when a significant part of the entries of $A_0$ are
zero, that $A_0$ is low rank and that the non-zero singular values of
$A_0$ are well-spread.

\section{Results}
\label{sec:main_results}

We will use the following notation: for a matrix $A \in \cM_{m, T}$,
$\ve(A)$ denotes the vector of $\R^{mT}$ obtained by stacking its
columns into a single vector. Note that this is an isometry between
$(\cM_{m, T}, \norm{\cdot}_{S_2})$ and $(\R^{mT},
|\cdot|_{\ell_2^{mT}})$ since $\inr{A, B} = \inr{\ve A, \ve B}$. We
introduce also the $\ell_\infty$ norm $\norm{A}_\infty = \max_{p, q}
|A_{p, q}|$.  Let us recall that for $\alpha \geq1$, the
$\psi_\alpha$-norm of a random variable $Z$ is given by
$\norm{Z}_{\psi_\alpha} := \inf \{ c > 0 : \E[ \exp( |Z|^\alpha/
c^\alpha)) ] \leq 2 \}$ and a similar norm can be defined for $0<\alpha<1$ (cf. \cite{LT:91}).

\subsection{Assumptions and examples}
\label{sec:assumptions}
The first assumption concers the ``covariate'' matrix $X$.
\begin{assumption}[Matrix $X$]
  \label{ass:X}
  There are positive constants $b_{X,\infty}, b_{X,\ell_\infty}$ and
  $b_{X,2}$ such that $\norm{X}_{S_\infty}\leq b_{X, \infty}$,
  $\norm{X}_{\infty}\leq b_{X, \ell_\infty}$ and $\norm{X}_{S_2} \leq
  b_{X, 2}$ almost surely. Moreover, we assume that the ``covariance
  matrix''
  \begin{equation*}
    \Sigma := \E[ \ve X (\ve X)^\top]
  \end{equation*}
  is invertible.
\end{assumption}
 This assumption  is met in the matrix completion and the
multitask-learning problems:
\begin{enumerate}
\item In the \textit{matrix completion problem}, the matrix $X$ is
  uniformly distributed over the set $\{ e_{p, q} : 1 \leq p \leq m, 1
  \leq q \leq T \}$ (see Section~\eqref{sec:motivations}), so in this
  case $\Sigma = (mT)^{-1} I_{m \times T}$ and $b_{X, 2} = b_{X,
    \infty} = b_{X, \ell_\infty} = 1$.
\item In the \textit{multitask-learning problem}, the matrix $X$ is
  uniformly distributed in
  $\{A_j(x_{j,s}):j=1,\ldots,T;s=1,\ldots,k_j\}$, where
  $(x_{j,s}:j=1,\ldots,T;s=1,\ldots,k_j)$ is a family of vectors in
  $\R^m$ and for any $j=1,\ldots,T$ and $x\in\R^m$,
  $A_j(x)\in\cM_{m,T}$ is the matrix having the vector $x$ for $j$-th
  column and zero everywhere else. So, in this case $\Sigma$ is equal
  to $T^{-1}$ times the $mT\times mT$ block matrix with $T$ diagonal
  blocks of size $m \times m$ made of the $T$ matrices
  $k_{j}^{-1}\sum_{i=1}^{k_j} x_{j,s}x_{j,s}^\top$ for $j=1,\ldots,T$.

  If we assume that the smallest singular values of the matrices
  $k_j^{-1}\sum_{i=1}^{k_j} x_{j,s}x_{j,s}^\top \in \cM_{m,m}$ for
  $j=1,\ldots,T$ are larger than a constant $\sigma_{\min}$ (note that
  this implies that $k_j \geq m$), then $\Sigma$ has its smallest
  singular value larger than $\sigma_{\min}T^{-1}$, so it is
  invertible. Moreover, if the vectors $x_{j, s}$ are normalized in
  $\ell_2$, then one can take $b_{X, \infty} = b_{X, \ell_\infty} =
  b_{X, 2} = 1$.
\end{enumerate}

The next assumption deals with the regression function of $Y$ given $X$. It is standard in regression analysis.

\begin{assumption}[Noise]
  \label{ass:Y}
  There are positive constants $b_Y, b_{Y, \infty}, b_{Y,\psi_2},
  b_{Y, 2}$ such that $\norm{Y-\E(Y|X)}_{\psi_2} \leq b_{Y, \psi_2}$,
  $\norm{\E(Y|X)}_{L_\infty} \leq b_{Y, \infty}$, $\E[(Y-\E(Y|X))^2 |
  X ] \leq b_{Y, 2}^2$ almost surely and $\E Y^2 \leq b_Y^2$.
\end{assumption}
In particular, any model $Y = \inr{A_0, X} + \varepsilon$, where
$\norm{A_0}_{S_\infty} < +\infty$ and $\varepsilon$ is a sub-gaussian
noise satisfies Assumption~\ref{ass:Y}.  Note that by using the whole
strength of Talagrand's concentration inequality on product spaces for
$\psi_\alpha$ ($0<\alpha\leq1$) random variables obtained in
\cite{MR2424985}, other type of tail decay of the noise could be
considered (yet leading to slower decay of the residual term)
depending on this assumption.


\subsection{Main results}
\label{sec:main-results}

In this section we state our main results. We give sharp oracle
inequalities for the penalized empirical risk minimization procedure
\begin{equation}
  \label{eq:PERM}
  \hat A_n \in \argmin_{A \in \cM_{m,T}} \Big\{ \frac 1n \sum_{i=1}^n (Y_i
  - \prodsca{X_i}{A})^2 + \pen(A) \Big\},
\end{equation}
where $\pen(A)$ is a penalty function which will be either a
pure $\norm{\cdot}_{S_1}$ penalization, or a ``matrix elastic-net''
penalization $\norm{\cdot}_{S_1} + \norm{\cdot}_{S_2}^2$ or other
penalty functions involving the $\norm{\cdot}_1$ norm.


\begin{theorem}[Pure $\norm{\cdot}_{S_1}$ penalization]
  \label{thm:oracle-S1}
  There is an absolute constants $c > 0$ such that the following
  holds. Let Assumptions~\ref{ass:X} and \ref{ass:Y} hold, and let
  $x>0$ be the some fixed confidence level. Consider any
  \begin{equation*}
    \hat A_n \in \argmin_{A \in \cM_{m, T}} \big\{ R_n(A) + \lambda_{n, x}
    \norm{A}_{S_1} \big\},
  \end{equation*}
  for
  \begin{equation*}
    \lambda_{n, x} = c_{X, Y} \frac{(x + \log n) \log n}{\sqrt n},
  \end{equation*}
  where $c_{X, Y} := c (1 + b_{X, 2}^2 + b_Y b_X + b_{Y, \psi_1}^2 +
  b_{Y, \infty}^2 + b_{Y, 2}^2 + b_{X, \infty}^2)$. Then one has, with
  a probability larger than $1 - 5e^{-x}$, that
  \begin{equation*}
    R(\hat A_n) \leq \inf_{A \in \cM_{m, T}} \big\{R(A) + \lambda_{n, x} (1 + \norm{A}_{S_1}) \big\}.
  \end{equation*}
\end{theorem}
Note that the residue that we obtain is of the form $\norm{A_0}_{S_1}/\sqrt{n}$. In particular, this residual term is not deteriorated if $A_0$ is of full rank but close to a low rank matrix. Classical residue involving the rank of $A_0$ are useless in this situation. It is also still meaningful when the quantity  $m+T$ becomes large compare to $n$. This is not the case of the residue of the form  $r(m+T)/n$ obtained previously for the same procedure (for other risks and under other - stronger -  assumptions).

We now state three sharp oracle inequalities for procedures of the form (\ref{eq:PERM}) where the penalty function is a mixture of norms.

\begin{theorem}[Matrix Elastic-Net]
  \label{thm:oracle-S1-S2}
  There is an absolute constant $c > 0$ such that the following holds.
  Let Assumptions~\ref{ass:X} and~\ref{ass:Y} hold. Fix any $x > 0$,
  $r_1, r_2 > 0$, and consider
  \begin{equation*}
    \hat A_n \in \argmin_{A \in \cM_{m, T}} \big\{ R_n(A)+ \lambda_{n, x}
    (r_1 \norm{A}_{S_1}  + r_2 \norm{A}_{S_2}^2) \Big\},
  \end{equation*}
  where 
  \begin{equation*}
    \lambda_{n, x} = c_{X, Y} \frac{\log n}{\sqrt
      n} \Big( \frac{1}{r_1} + \frac{(x + \log n) \log
      n}{r_2 \sqrt n} \Big), 
  \end{equation*}
  where $c_{X, Y} = c (1 + b_{X, 2}^2 + b_{X, 2} b_Y + b_{Y, \psi_1}^2
  + b_{Y, \infty}^2 + b_{Y, 2}^2)$. Then one has, with a probability
  larger than $1 - 5e^{-x}$, that
  \begin{equation*}
    R(\hat A_n) \leq \inf_{A \in \cM_{m, T}} \big\{R(A)+ \lambda_{n, x} (1 + r_1 \norm{A}_{S_1}  +  r_2 \norm{A}_{S_2}^2) \big\}.
  \end{equation*}
\end{theorem}

\begin{theorem}[$\norm{\cdot}_{S_1} + \norm{\cdot}_{1}$ penalization]
  \label{thm:oracle-S1L1}
  There is an absolute constant $c > 0$ such that the following
  holds. Let Assumptions~\ref{ass:X} and \ref{ass:Y} hold. Fix any $x,
  r_1, r_3 > 0$, and consider
  \begin{equation*}
    \hat A_n \in \argmin_{A \in \cM_{m, T}} \big\{R_n(A) + \lambda_{n, x} (r_1 \norm{A}_{S_1}  + r_3\norm{A}_{1}) \big\}
  \end{equation*}
  for
  \begin{equation*}
    \lambda_{n, x} := c_{X, Y} \Big( \frac{1}{r_1} \wedge
    \frac{\sqrt{\log(m T)}}{r_3}\Big) \frac{(x + \log n) (\log n)^{3/2}}{\sqrt
      n},
  \end{equation*}
  where $c_{X, Y} = c (1 + b_{X, 2}^2 + b_{X, 2} b_Y + b_{Y, \psi_1}^2
  + b_{Y, \infty}^2 + b_{Y, 2}^2 + b_{X, \infty}^2 + b_{X,
    \ell_\infty}^2)$. Then one has, with a probability larger than $1
  - 5e^{-x}$, that
  \begin{equation*}
    R(\hat A_n) \leq \inf_{A \in \cM_{m, T}} \big\{R(A) + \lambda_{n, x} (1 + r_1 \norm{A}_{S_1}  + r_3\norm{A}_{1})) \big\}.
  \end{equation*}
\end{theorem}

\begin{theorem}[$\norm{\cdot}_{S_1} + \norm{\cdot}_{S_2}^2 +
  \norm{\cdot}_{1}$ penalization]
  \label{thm:oracle-S1S2L1}
  There is an absolute constant $c > 0$ such that the following
  holds. Let Assumptions~\ref{ass:X} and \ref{ass:Y} hold. Fix any $x,
  r_1, r_2, r_3 > 0$, and consider
  \begin{equation*}
    \hat A_n \in \argmin_{A \in \cM_{m, T}} \big\{R_n(A) + \lambda_{n, x} (r_1 \norm{A}_{S_1} + r_2\norm{A}_{S_2}^2  + r_3 \norm{A}_{1}) \Big\}
  \end{equation*}
  for
  \begin{align*}
    \lambda_{n, x} := c_{X, Y} \frac{(\log n)^{3/2}}{\sqrt n} \Big(
    \frac{1}{r_1} \wedge \frac{\sqrt{\log(m T)}}{r_3} + \frac{x + \log
      n}{r_2 \sqrt n} \Big),
  \end{align*}
  where $c_{X, Y} = c (1 + b_{X, 2}^2 + b_{X, 2} b_Y + b_{Y, \psi_1}^2
  + b_{Y, \infty}^2 + b_{Y, 2}^2)$. Then one has, with a probability
  larger than $1 - 5e^{-x}$, that
  \begin{equation*}
    R(\hat A_n) \leq \inf_{A \in \cM_{m, T}} \big\{ R(A) + \lambda_{n, x} (1 + r_1 \norm{A}_{S_1} + r_2\norm{A}_{S_2}^2  + r_3 \norm{A}_{1})) \big\}.
  \end{equation*}
\end{theorem}
The parameters $r_1, r_2$ and $r_3$ in the above procedures are
completely free and can depend on $n, m$ and $T$. Intuitively, it is
clear that $r_2$ should be smaller than $r_1$ since the $\norm{\cdot}_{S_2}$ term is
 used for ``regularization'' of the non-zero singular values only, while the term
$\norm{\cdot}_{S_1}$ makes $\hat A_n$ of low rank, as for the elastic-net for
vectors (see~\cite{MR2137327}). Indeed, for the $\norm{\cdot}_{S_1} +
\norm{\cdot}_{S_2}^2$ penalization, any choice of $r_1$ and $r_2$ such
that $r_2 = r_1 \log n / \sqrt{n}$ leads to a residual term smaller than
\begin{equation*}
  c_{X, Y} (1 + x + \log n) \Big( \frac{(\log n)^2}{r_2 n} + \frac{\log n}{\sqrt
    n} \norm{A}_{S_1} + \frac{(\log n)^2}{n} \norm{A}_{S_2}^2 \Big).
\end{equation*}
Note that the rate related to $\norm{A}_{S_1}$ is (up to logarithms)
$1 / \sqrt n$ while the rate related to $\norm{A}_{S_2}^2$ is
$1/n$. The choice of $r_3$ depends on the number of zeros in the
matrix. Note that in the $\norm{\cdot}_{S_1} + \norm{\cdot}_1$ case,
any choice $1 \leq r_3 \leq r_1$ entails a residue smaller than
\begin{equation*}
  c_{X, Y} \frac{(x + \log n) \log n}{\sqrt n} (1 +
  \norm{A}_{S_1}  + \norm{A}_{1}),
\end{equation*}
which makes again the residue independent of $m$ and $T$.

Note that, in the matrix completion case, the term $\sqrt{\log mT}$ can be removed from the regularization (and thus the residual) term thanks to the second statement of Proposition~\ref{prop:complexity} below.

\section{Proof of the main results}

\subsection{Some definitions}

For any $r, r_1, r_2, r_3 \geq 0$, we consider the ball
\begin{equation}
  \label{eq:Br}
  B_{r, r_1, r_2, r_3} := \{ A \in \cM_{m, T} : r_1 \norm{A}_{S_1} + r_2
  \norm{A}_{S_2}^2 + r_3 \norm{A}_1 \leq r \},
\end{equation}
and we denote by $B_{r, 1} = B_{r, 1, 0, 0}$ the \emph{nuclear norm}
ball, by $B_{r, r_1, r_2} = B_{r, r_1, r_2, 0}$ the \emph{elastic-net}
ball. In what follows, $B_r$ will be either $B_{r, 1}$, $B_{r, r_1,
  r_2}$, $B_{r, r_1, r_2, r_3}$ or $B_{r, r_1, 0, r_3}$, depending on
the penalization. We consider an oracle matrix in $B_r$ given by:
\begin{equation*}
  A_r^* \in \argmin_{A \in B_r} \E (Y - \inr{X, A})^2
\end{equation*}
and the following excess loss function over $B_r$ defined for any $A \in B_r$ by
\begin{equation*}
  \cL_{r,A}(X, Y) := (Y - \prodsca{X}{A})^2 - (Y -
  \prodsca{X}{A_r^*})^2.
\end{equation*}
Define also the excess loss functions class
\begin{equation}
  \label{eq:Excess-Loss-Class}
  \cL_r := \{ \cL_{r, A} : A \in B_r \}.
\end{equation}
The star-shaped-hull at $0$ of $\cL_r$ is given by
\begin{equation*}
  V_r := \Star(\cL_r, 0) = \{ \alpha \cL_{r,A} : A \in B_r \mbox{ and }
  0 \leq \alpha \leq 1 \}
\end{equation*}
and its localized set at level $\lambda>0$
\begin{equation}
  \label{eq:V_rlambda}
  V_{r, \lambda} := \{g\in V_r:\E g\leq \lambda\}.
\end{equation}
The proof of Theorems~\ref{thm:oracle-S1} to~\ref{thm:oracle-S1S2L1}
rely on the \emph{isomorphic penalization method}, introduced by
P. Bartlett, S. Mendelson and J. Neeman (cf. \cite{B:08},
\cite{Mendelson08regularizationin} and \cite{BM:08}). It has improved
several results on penalized empirical risk minimization procedures
for the Lasso (cf. \cite{BM:08}) and for regularization in reproducing
kernel Hilbert spaces (cf. \cite{Mendelson08regularizationin}). This
approach relies on a sharp analysis of the complexity of the set
$V_{r, \lambda}$. Indeed, an important quantity appearing in learning
theory is the maximal deviation of the empirical distribution around
its mean uniformly over a class of function. If $V$ is a class of
functions, we define the supremum of the deviation of the empirical
mean around its actual mean over $V$ by
\begin{equation*}
  \|P_n - P\|_{V} = \sup_{h\in V} \Big|\frac{1}{n} \sum_{i=1}^n h(X_i,
  Y_i) - \E h(X, Y) \Big|.
\end{equation*}

\subsection{On the importance of convexity}
An important parameter driving the quality of concentration of $\|P_n
- P\|_{V}$ to its expectation is the so-called Bernstein's parameter
(cf. \cite{BM:06}). We are studying this parameter in our context without introducing a formal definition of this quantity. 

For every
matrix $A \in \cM_{m, T}$, we consider the random variable $f_A :=
\prodsca{X}{A}$ and the following subset of $L_2$:
\begin{equation}
  \label{eq:f_W}
  \cC_r := \{ f_A : A \in B_r\},
\end{equation}
where $B_r = B_{r, r_1, r_2, r_3}$ is given by~\eqref{eq:Br}. Because
of the convexity of the norms $\norm{\cdot}_{S_1}$,
$\norm{\cdot}_{S_2}$ and $\norm{\cdot}_{1}$, the set $\cC_r$ is
convex, for any $r, r_1, r_2, r_3 \geq 0$. Now, consider the following
minimum
\begin{equation}
  \label{eq:minimum-fr}
  f^*_r \in \argmin_{f \in \cC_r} \norm{Y - f}_{L_2}
\end{equation}
and
\begin{equation}
  \label{eq:parameter-Cr}
  C_r := \min\Big(b_{X,\infty}\frac{r}{r_1},
  b_{X,2}\sqrt{\frac{r}{r_2}}, b_{X, \ell_\infty} \frac{r}{r_3} \Big),
\end{equation}
with the convention $1 / 0 = +\infty$.
\begin{lemma}[Bernstein's parameter]
  \label{lem:convexity}
  Let assumptions~\ref{ass:X} and~\ref{ass:Y} hold.  There is a unique
  $f^*_r$ satisfying~\eqref{eq:minimum-fr} and a unique $A_r^*\in B_r$
  such that $f_r^* = f_{A_r^*}$. Moreover, any $ A \in B_r$ satisfies
  \begin{equation*}
    \E\cL_{r,A} \geq \E
    \prodsca{X}{A - A_r^*}^2,
  \end{equation*}
  and the class $\cL_r$ satisfies the following Bernstein's condition: for all $A \in B_r$
  \begin{equation*}
    \E\cL_{r, A}^2 \leq  4 ( b_{Y, 2}^2 + (b_{Y,\infty} + C_r)^2 ) 
    \E\cL_{r, A}.
  \end{equation*}
\end{lemma}

\begin{proof}
  By convexity of $\cC_r$ we have $\prodsca{Y - f^*_r}{f -
    f_r^*}_{L^2} \leq 0$ for any $f \in \cC_r$. Thus, we have, for any
  $f \in \cC_r$
  \begin{equation}
    \label{eq:one-side-bern1}
    \norm{Y - f}_{L_2}^2 - \norm{Y - f^*_r}^2_{L_2} =
    2\prodsca{f^*_r - f}{Y - f^*_r} + \norm{f - f^*_r}_{L_2}^2 \geq
    \norm{f - f^*_r}_{L_2}^2.
  \end{equation}
  In particular, the minimum is unique. Moreover, $\cC_r$ is a closed
  set and since $\Sigma$ is invertible under Assumption~\ref{ass:X},
  there is a unique $A_r^* \in B_r$ such that $f^*_r = f_{A^*_r}$. By
  the trace duality formula and Assumption~\ref{ass:X}, we have, for
  any $A \in B_{r, r_1, r_2, r_3}$:
  \begin{align*}
    &|f_A| \leq \norm{X}_{S_\infty} \norm{A}_{S_1} \leq b_{X, \infty}
    \frac{r}{r_1}, \quad |f_A| \leq \norm{X}_{S_2} \norm{A}_{S_2} \leq
    b_{X, 2} \sqrt{\frac{r}{r_2}}, \\
    &\text{ and } |f_A| \leq \norm{X}_{\infty} \norm{A}_{1} \leq b_{X,
      \ell_\infty} \frac{r}{r_3}
  \end{align*}
  almost surely, so that $|f_A| \leq C_r$ for any $A \in B_{r}$ a.s..
  Moreover, for any $A \in B_r$:
  \begin{equation}
    \label{eq:excess-decomposition}
    \cL_{r,A} = 2(Y - \E(Y|X)) \inr{X,A_r^* - A} + (
    2\E(Y|X) - \inr{A + A_r^*,X}) \inr{X,A_r^* - A}.
  \end{equation}
  Thus, using Assumption~\ref{ass:Y}, we obtain
  \begin{align*}
    \nonumber \E \cL_{r, A}^2 &= \E \big[4(Y - \E(Y|X))^2
    \inr{X,A-A_r^*}^2 +
    (2\E(Y|X)-\inr{X,A+A_r^*})^2 \inr{X,A-A_r^*}^2\big] \\
    \nonumber &\leq 4 \E\big[\inr{X,A - A_r^*}^2 \E \big[(Y -
    \E(Y|X))^2 | X \big] \big] + 4 (b_{Y,\infty} + C_r)^2 \E\inr{X, A-A_r^*}^2 \\
    &\leq 4 (b_{Y,2}^2 + (b_{Y\infty} + C_r)^2 ) \E\inr{X,A-A_r^*}^2,
  \end{align*}
  which concludes the proof using (\ref{eq:one-side-bern1}).
\end{proof}

\subsection{The isomorphic property of the excess loss functions class}

The \emph{isomorphic property} of a functions class has been
introduced in \cite{Mendelson08regularizationin} and is a consequence
of Talagrand's concentration inequality (cf.~\cite{MR1419006}) applied
to a localization of the functions class together with the Bernstein
property of this class (here this property was studied in
Lemma~\ref{lem:convexity}). We recall here the argument in our special
case.
\begin{theorem}[\cite{MR2240689}]
  \label{cor:isomorphy}
  There exists an absolute constant $c > 0$ such that the following
  holds. Let Assumptions~\ref{ass:X} and~\ref{ass:Y} hold. Let $r > 0$
  and $\lambda(r)>0$ be such that
  \begin{equation*}
    \E\|P_n - P\|_{V_{r, \lambda(r)}} \leq \frac{\lambda(r)}{8}.
  \end{equation*}
  Then, with probability larger than $1 -4 e^{-x}$: for all $A \in B_r$
  \begin{equation*}
    \frac 12 P_n \cL_{r, A} - \rho_n(r, x) \leq P \cL_{r, A} \leq 2
    P_n \cL_{r, A} + \rho_n(r, x),
  \end{equation*}
  where
  \begin{equation*}
    \rho_n(r, x) := c \Big( \lambda(r) +
    \big[b_{Y, \psi_1} + b_{Y,\infty} + b_{Y,2} + C_r \big]^2
    \Big(\frac{x \log n}{n}\Big) \Big),
  \end{equation*}
  and $C_r$ has been introduced in~\eqref{eq:parameter-Cr}.
\end{theorem}

\begin{proof}
  We follow the line of \cite{MR2240689}. Let $\lambda>0$ and
  $x>0$. Thanks to \cite{MR2424985}, with probability larger than
  $1-4\exp(-x)$,
  \begin{equation}\label{eq:Adamcjak}
    \norm{P - P_n}_{V_{r,\lambda}} \leq 2
    \E\norm{P-P_n}_{V_{r,\lambda}} + c_1 \sigma(V_{r,\lambda})
    \sqrt{\frac x n} + c_2 b_n(V_{r,\lambda})\frac x n
  \end{equation}
  where, by using the Bernstein's properties of $\cL_r$
  (cf. Lemma~\ref{lem:convexity})
  \begin{align}
    \label{eq:Variance-term-adamcjak}
    \nonumber \sigma^2(V_{r,\lambda}) &:= \sup_{g\in V_{r,\lambda}}
    \var(g) \leq \sup \Big(\E(\alpha \cL_{r,A})^2 : 0 \leq \alpha \leq
    1, A \in B_r, \E( \alpha \cL_{r,A}) \leq \lambda \Big) \\
    \nonumber &\leq \sup \Big( 4 ( b_{Y,2}^2 + (b_{Y,\infty} + C_r)^2)
    \E(\alpha \cL_{r, A}) : 0 \leq \alpha \leq 1, A \in
    B_r, \E(\alpha\cL_{r,A}) \leq \lambda\Big) \\
    &\leq 4 ( b_{Y,2}^2 + (b_{Y,\infty} + C_r)^2 ) \lambda,
  \end{align}
  and using Pisier's inequality (cf. \cite{vanderVaartWellner}):
  \begin{align}
    \label{eq:psi1-term-adamcjak}
    \nonumber &b_n(V_{r,\lambda}) := \Big\| \max_{1\leq i\leq n}
    \sup_{g\in V_{r,\lambda}} g(X_i,Y_i) \Big\|_{\psi_1} \leq \log n
    \Big\| \sup_{g\in
      V_{r,\lambda}} g(X, Y) \Big\|_{\psi_1} \\
    \nonumber &= \log n \Big\| \sup \Big(\alpha(2Y - \inr{X, A +
      A_r^*}) \inr{X, A_r^* - A} : 0 \leq \alpha \leq 1, A \in
    B_r \Big) \Big\|_{\psi_1} \\
    &\leq 4(\log n) (b_{Y, \psi_1} + b_{Y,\infty} + C_r ) C_r,
  \end{align}
  where we used decomposition~\eqref{eq:excess-decomposition} and
  Assumption~\ref{ass:Y} together with the uniform bound $|\inr{A,X}|
  \leq C_r$ holding  for all $A \in B_r$.
  
  Moreover, for any $\lambda>0$, $V_{r,\lambda}$ is star-shaped so $G
  : \lambda \mapsto \E\norm{P-P_n}_{V_{r,\lambda}}/\lambda$ is
  non-increasing. Since $G(\lambda(r))\leq1/8$ and
  $\rho_n(r,x)\geq\lambda(r)$, we have
  \begin{equation*}
    \E\norm{P-P_n}_{V_{r,\rho_n(r,x)}}\leq \rho_n(r,x)/8,
  \end{equation*}
  which yields, in Equation~\eqref{eq:Adamcjak} together with the
  variance control of Equation~(\ref{eq:Variance-term-adamcjak}) and
  the control of Equation~(\ref{eq:psi1-term-adamcjak}), that there
  exists an event $\Omega_0$ of probability measure greater than
  $1-4\exp(-x)$ such that, on $\Omega_0$,
  \begin{align}
    \label{eq:Admaclak2}
    \nonumber \norm{P-P_n}_{V_{r,\rho_n(r,x)}} &\leq
    \frac{\rho_n(r,x)}{4} + c_1 (b_{Y, \infty} + b_{Y,2} + C_r )
    \sqrt{\frac{\rho_n(r,x) x}{n}} \\
    \nonumber
    &+ c_2 ( b_{Y,\psi_1} + b_{Y,\infty} + C_r ) C_r\frac{x\log n}{n} \\
    & \leq\frac{\rho_n(r,x)}{2}
  \end{align}
  in view of the definition of $\rho_n(r,x)$.  In particular, on
  $\Omega_0$, for every $A\in B_r$ such that
  $P\cL_{r,A}\leq\rho_n(r,x)$, we have $|P\cL_{r,A} -
  P_n\cL_{r,A}|\leq\rho_n(r,x)/2$. Now, take $A\in B_r$ such that
  $P\cL_{r,A}=\beta>\rho_n(r,x)$ and set
  $g=\rho_n(r,x)\cL_{r,A}/\beta$. Since $g\in V_{r,\rho_n(r,x)}$,
  Equation~(\ref{eq:Admaclak2}) yields, on $\Omega_0$, $|Pg-P_ng|\leq
  \rho_n(r,x)/2<\beta/2$ and so $(1/2)P_n\cL_{r,A}\leq P\cL_{r,A}\leq
  (3/2)P_n\cL_{r,A}$ which concludes the proof.
\end{proof}

A function $r \mapsto \lambda(r)$ such that $\E\|P_n - P\|_{V_{r,
    \lambda(r)}} \leq \lambda(r)/8$ is called an \emph{isomorphic
  function} and is directly connected to the choice of the
penalization used in the procedure which was introduced in
Section~\ref{sec:main_results}. The computation of this function is
related to the complexity of Schatten balls, computed in the next
section.

\subsection{Complexity of Schatten balls}

The \emph{generic chaining} technique (see \cite{Talagrand:05}) is a
powerful technique for the control of the supremum of empirical
processes. For a subgaussian process, such a control is achieved using
the $\gamma_2$ functional recalled in the next definition.
\begin{definition}[\cite{Talagrand:05}] Let $(F, d)$ be a metric
  space. We say that $(F_j)_{j \geq 0}$ is an \emph{admissible
  sequence} of partitions of $F$ if $|F_0| = 1$ and $|F_j| \leq
  2^{2^j}$ for all $j\geq1$. The $\gamma_2$ functional is defined by
  \begin{equation*}
    \gamma_2(F, d) = \inf_{(F_j)_j} \sup_{f \in F} \sum_{j \geq 0} 2^{j / 2}
    d(f, F_j),
  \end{equation*}
  where the infimum is taken over all admissible sequence $(F_j)_{j\geq1}$ of $F$.
\end{definition}

A classical upper bound on the $\gamma_2$ functional is the  Dudley's entropy integral:
\begin{equation}
  \label{eq:Dudley-entropy}
  \gamma_2(F,d)\leq c_0\int_0^\infty\sqrt{\log N(F,d,\epsilon)}d\epsilon,
\end{equation}where
$N(B, \norm{\cdot}, \varepsilon)$ is the minimal number of balls with respect to the metric $d$ of radius $\epsilon$ needed to cover $B$. When $B$ enjoys some convexity properties, this bound can be improved. Let $(E, \norm{\cdot})$ be a Banach space. We denote by $B(E)$ its
unit ball. We say that $(E,\norm{\cdot})$ is $2$-convex if there
exists some $\rho>0$ such that for all $x,y\in B(E)$, we have
\begin{equation*}
  \norm{x+y}\leq 2-2\rho\norm{x-y}^2.
\end{equation*}
In the case of $2$-convex bodies, the following theorem gives an upper
bound on the $\gamma_2$ functional that can improve the one given by
Dudley's entropy integral.
\begin{theorem}[\cite{Talagrand:05}]
  \label{theo:complexity-Talagrand-2convexBodies}
  For any $\rho>0$, there exists $c(\rho)>0$ such that if
  $(E,\norm{\cdot})$ is a $2$-convex Banach space and $\norm{\cdot}_E$
  is another norm on $E$, then
  \begin{equation*}
    \gamma_{2}(B(E),\norm{\cdot}_E) \leq c(\rho) \Big(\int_0^\infty
    \epsilon \log N(B(E),\norm{\cdot}_E, \epsilon) d\epsilon
    \Big)^{1/2}.
  \end{equation*}
\end{theorem} The  generic chaining technique provides  the following upper bound on Gaussian processes.
\begin{theorem}[\cite{Talagrand:05}]
  \label{theo:Generic-chaining-sous-gaussien}
  There is an absolute constants $c > 0$ such that the following
  holds. If $(Z_f)_{f \in F}$ is a subgaussian process for some metric
  $d$ (i.e. $\norm{Z_f-Z_g}_{\psi_2}\leq c_0 d(f,g)$ for all $f,g\in F$) and if $f_0 \in F$, then one has
  \begin{equation*}
    \E \sup_{f \in F} |Z_f - Z_{f_0}| \leq c \gamma_2(F, d).
  \end{equation*}
\end{theorem}

The metric  used to measure the complexity of the excess loss classes we are working on is an empirical one defined for any $A\in\cM_{m,T}$ by
\begin{equation}
  \label{eq:ell-infty-n-metric}
  \norm{A}_{\infty,n} := \max_{1\leq i\leq n}|\inr{X_i,A}|.
\end{equation}This metric comes out of the so-called $L_{\infty,n}$-method of M.~Rudelson introduced in~\cite{MR1694526} and first used in learning theory in \cite{Mendelson08regularizationin}.
We denote by $B(S_p)$ the unit ball of the Banach space $S_p$ of
matrices in $\cM_{m, T}$ endowed with the Schatten norm
$\norm{\cdot}_{S_p}$. We denote also by $B_1$ the unit ball of
$\cM_{m, T}$ endowed with the $\ell_1$-norm $\norm{\cdot}_{1}$. In the following, we compute the complexity of the balls $B(S_1),B(S_2)$ and $B_1$ with respect to the empirical metric $\norm{\cdot}_{\infty,n}$.
\begin{proposition}
  \label{prop:complexity}There exists an absolute constant $c>0$ such that the following holds.  Assume that $\norm{X_i}_{S_2},\norm{X_i}_\infty\leq1$ for all $i=1, \ldots, n$. Then, we have
  \begin{equation*}
    \gamma_2(r B(S_1), \norm{\cdot}_{\infty,n}) \leq \gamma_2(r
    B(S_2), \norm{\cdot}_{\infty,n}) \leq c r \log n
  \end{equation*}
  and
  \begin{equation*}
    \gamma_2(rB_1, \norm{\cdot}_{\infty,n}) \leq c r (\log n)^{3/2} \sqrt{\log(mT)}.
  \end{equation*}
  Moreover, if we assume that $X_1, \ldots, X_n$ have been obtained  in the matrix
  completion model then 
  \begin{equation*}
    \gamma_2(rB_1, \norm{\cdot}_{\infty,n}) \leq c r (\log n)^{3/2}.
  \end{equation*}
\end{proposition}

\begin{proof}
  The first inequality is obvious since $B(S_1) \subset B(S_2)$. By
  using Dual Sudakov's inequality (cf. \cite{MR817602}), we have for all $\epsilon>0$,
  \begin{equation*}
    \log N(B(S_2), \norm{\cdot}_{\infty,n}, \epsilon) \leq c_0
    \Big(\frac{\E\norm{G}_{\infty,n}}{\epsilon}\Big)^2,
  \end{equation*}where $G$ is a $m\times T$ matrix with i.i.d. standard Gaussian random variables for entries.   A Gaussian maximal inequality and the fact that $\norm{X_i}_{S_2} \leq 1$ for all $i=1,\ldots,n$ provides $\E\norm{G}_{\infty,n}\leq c_1\sqrt{\log
    n}$, hence
  \begin{equation*}    
    \log N(B(S_2), \norm{\cdot}_{\infty,n}, \epsilon) \leq
    \frac{c_2\log n}{\epsilon^2}.
  \end{equation*}
 Denote by $B_{\infty, n}$ the unit ball of $(\cM_{m,
  T}, \norm{\cdot}_{\infty, n})$ in $V_n = \Span(X_1, \ldots, X_n)$,
the linear subspace of $\cM_{m,T}$ spanned by $X_1, \ldots, X_n$. The volumetric argument provides
  \begin{align*}
    \log N(B(S_2),\norm{\cdot}_{\infty,n}, \epsilon) &\leq \log
    N(B(S_2),\norm{\cdot}_{\infty,n},\eta) + \log
    N(\eta B_{\infty,n},\epsilon B_{\infty,n}) \\
    &\leq \frac{c_2\log n}{\eta^2}+n\log\Big(\frac{3\eta}{\epsilon}\Big)
  \end{align*}
  for any $\eta \geq \epsilon>0$. Thus, for $\eta_n = \sqrt{\log n /
    n}$, we have, for all $0<\epsilon \leq \eta_n$
  \begin{equation*}
    \log N(B(S_2),\norm{\cdot}_{\infty,n},\epsilon) \leq c_3
    n \log \Big( \frac{3 \eta_n}{\epsilon} \Big).
  \end{equation*}Since $B(S_2)$ is the unit ball of a Hilbert space, it is $2$-convex. We can thus apply Theorem~\ref{theo:complexity-Talagrand-2convexBodies} to obtain the following upper bound
  \begin{equation*}
    \gamma_2(r B(S_2), \norm{\cdot}_{\infty,n})\leq c_4 r \log n.    
  \end{equation*}
  
 Now, we prove an upper bound on the complexity of $B_1$ with respect to $\norm{\cdot}_{\infty,n}$. Recall that $\ve : \cM_{m, T} \rightarrow \R^{m T}$ 
  concatenates the columns of a matrix into a single vector of size $m
  T$.  Obviously, $\ve$ is an isometry between $(\cM_{m, T},
  \norm{\cdot}_{S_2})$ and $(\R^{mT}, |\cdot|_{2})$, since
  $\prodsca{A}{B} = \prodsca{\ve(A)}{\ve(B)}$. Using this mapping, we
  see that, for any $\epsilon>0$,
  \begin{equation*}
    N(B_1, \norm{\cdot}_{\infty,n}, \epsilon) = N(b_1^{mT},
    |\cdot|_{\infty,n}, \epsilon)
  \end{equation*}
  where $b_1^{mT}$ is the unit ball of $\ell_1^{mT}$ and
  $|\cdot|_{\infty,n}$ is the pseudo norm on $\R^{mT}$ defined for any $x\in\R^{mT}$ by
  $|x|_{\infty,n} = \max_{1 \leq i \leq n}|\inr{y_i, x}|$ where
  $y_i={\rm vec}(X_i)$ for $i = 1, \ldots, n$. Note that $y_1, \ldots,
  y_n \in b_2^{mT}$, where $b_2^{mT}$ is the unit ball of
  $\ell_2^{mT}$. We use the Carl-Maurey's empirical method to compute
  the covering number $N(b_1^{mT},|\cdot|_{\infty, n},\epsilon)$ for 
  ``large scales'' of $\epsilon$ and the volumetric argument for ``small scales''. Let us begin with the
  Carl-Maurey's argument. Let $x \in b_1^{mT}$ and $Z$ be a random
  variable with values in $\{\pm e_1,\ldots,\pm e_{mT},0\}$ - where
  $(e_1, \ldots, e_{mT})$ is the canonical basis of $\R^{mT}$ - defined
  by $\P[Z = 0] = 1- |x|_1$ and for all $i=1,\ldots,mT$,
  \begin{equation*}
    \P[Z={\rm sign}(x_i)e_i] = |x_i|.
  \end{equation*}
  Note that $\E Z = x$. Let $s\in\N-\{0\}$ to be defined later and
  take $s$ i.i.d. copies of $Z$ denoted by $Z_1,\ldots,Z_s$. By the
  Gin{\'e}-Zinn symmetrization argument and the fact that Rademacher
  processes are upper bounded by Gaussian processes, we have
  \begin{equation}\label{eq:CM-argument1}
    \E \Big| \frac{1}{s} \sum_{i=1}^s Z_i - \E Z \Big|_{\infty, n} \leq c_0
    \E \Big| \frac{1}{s} \sum_{i=1}^s g_i Z_i \Big|_{\infty, n} \leq
    c_1 \sqrt{\frac{\log n}{s}}
  \end{equation}
  where the last inequality follows by a Gaussian maximal inequality
  and the fact that $|y_i|_2 \leq 1$. Take $s \in \N$ to be the
  smallest integer such that $\epsilon\geq c_1\sqrt{(\log
    n)/s}$. Then, the set
  \begin{equation}
    \label{eq:set-Maurey}
    \Big\{ \frac{1}{s} \sum_{i=1}^sz_i : z_1,\ldots,z_s \in \{\pm e_1,
    \ldots, \pm e_{mT}, 0\} \Big\}
  \end{equation}
  is an $\epsilon$-net of $b_1^{mT}$ with respect to $|\cdot|_{\infty,
    n}$. Indeed, thanks to (\ref{eq:CM-argument1}) there exists $\omega\in\Omega$ such that $|s^{-1}\sum_{i=1}^s Z_i(\omega)-x|_{\infty,n}\leq\epsilon$. This implies that there exists an element in the set (\ref{eq:set-Maurey}) which is $\epsilon$-close to $x$. Since the cardinality of the set introduced in (\ref{eq:set-Maurey}) is, according to \cite{MR810669}, at most
  \begin{equation*}
    \binom{2mT+s-1}{s}
    \leq \Big( \frac{e(2mT + s - 1)}{s} \Big)^s,
  \end{equation*}
  we obtain for any $\epsilon \geq \eta_n:= \big((\log n)(\log mT)/n\big)^{1/2} $ that
  \begin{equation*}
    \log N(b_1^{mT}, |\cdot|_{\infty, n} ,\epsilon)\leq s\log\Big( \frac{e(2mT + s - 1)}{s} \Big) \leq  \frac{c_2(\log n)  \log(mT)}{\epsilon^2},
  \end{equation*}
  and a volumetric argument gives
  \begin{equation*}
    \log N(b_1^{mT}, |\cdot|_{\infty, n} ,\epsilon) \leq c_3 n \log\Big(
    \frac{3 \eta_n}{\epsilon} \Big)
  \end{equation*}
  for any $0 < \epsilon \leq \eta_n$. Now we use the upper bound (\ref{eq:Dudley-entropy}) and compute the Dudley's entropy integral to obtain
  \begin{equation*}
    \gamma_2(rB_1, \norm{\cdot}_{\infty,n}) \leq c_4 r (\log n)^{3/2}
    \sqrt{\log(mT)}.
  \end{equation*}
 
 For the  ``matrix completion case'', we have 
 \begin{equation*}
   N(b_1^{mT},|\cdot|_{\infty,n},\epsilon)\leq N(b_1^n,\epsilon b_\infty^n)
 \end{equation*}where $ N(b_1^n,\epsilon b_\infty^n)$ is the minimal number of balls $\epsilon b_\infty^n$ needed to cover $b_1^n$. We use the following proposition from~\cite{MR732693} to compute $N(b_1^n,\epsilon b_\infty^n)$.
  \begin{proposition}[\cite{MR732693}]
    \label{prop:schutt}For any $\epsilon>0$, we have
    \begin{equation*}
      \log N(b_1^n,\epsilon b_\infty^n) \sim
      \begin{cases}
        0  & \text{ if } \epsilon \geq 1 \\
        \epsilon^{-1} \log \big(e n
        \epsilon \big) & \text{ if }
        n^{-1} \leq \epsilon\leq 1 \\
        n \log \big( 1/(\epsilon n) \big) &
        \text{ if } 0 < \epsilon \leq n^{-1}.
      \end{cases}
    \end{equation*}
  \end{proposition}
Then the result follows from (\ref{eq:Dudley-entropy}) and the computation of the Dudley's entropy integral using Proposition~\ref{prop:schutt}.
\end{proof}

\subsection{Computation of the isomorphic function}
\label{sec:Computation-of-the-isomorphic-function}
Introduce the ellipsoid
\begin{equation*}
  D := \{ A \in \cM_{m, T} : \E \prodsca{X}{A}^2 \leq 1 \}.
\end{equation*}
A consequence of Equation~(\ref{eq:one-side-bern1}) in
Lemma~\ref{lem:convexity} is the following inclusion, of importance in
what follows. Indeed, since $B_r$ is convex and symmetrical, one has:
\begin{equation}\label{eq:Inclusion-excess-loss-class-to-class}
  \{ A \in B_r : \E \cL_{r,A}\leq \lambda \} \subset
  A_r^* + K_{r, \lambda}, 
\end{equation}
where 
\begin{equation*}
  K_{r, \lambda} := 2 B_r \cap \sqrt \lambda D.
\end{equation*}
Hence, the complexity of $\{ A \in \cM_{m, T} : \cL_{r, A}
\in \cL_{r, \lambda} \}$ will be smaller than the complexity  of $B_r$
and $\sqrt \lambda D$. This will be of importance in the analysis
below. The next result provides an upper bound on the complexity of
$V_{r, \lambda}$ where we recall that
\begin{equation*}
  V_{r, \lambda} := \{\alpha \cL_{r,A} : 0 \leq \alpha \leq 1, A \in
  B_r, \E(\alpha \cL_{r,A}) \leq \lambda\}.
\end{equation*} 
From this statement we will derive corollaries that provide the shape
of the considered penalty functions.
\begin{proposition}
  \label{prop:penalty}
  There exists two absolute constants $c_1$ and $c_2$ such that the
  following holds. Let Assumptions~\ref{ass:X} and~\ref{ass:Y}
  hold. For any $r > 0$ and $\lambda>0$, we have
  \begin{equation*}
    \E\norm{P - P_n}_{V_{r,\lambda}} \leq c_1 \sum_{i \geq 0} 2^{-i}
    \phi_{n}(r, 2^{i+1} \lambda),
  \end{equation*}where
  \begin{equation*}
    \phi_n(r,\lambda):=c_2 \Big( U_n(K_{r, \lambda}) \sqrt \frac{\lambda}{n} +
    U_n(K_{r, \lambda}) \sqrt \frac{ R(A_r^*)}{n} +
    \frac{U_n(K_{r, \lambda})^2}{n} \Big),
  \end{equation*}for $K_{r,\lambda}=2B_r\cap\sqrt{\lambda}D$.
\end{proposition}

\begin{proof}
  Introduce $\cL_{r,\lambda}=\{\cL_{r,A}: A \in B_r,
  \E\cL_{r,A}\leq\lambda\}$. Using the Gin\'e-Zinn symmetrization
  \cite{MR757767} and the inclusion
  of~\eqref{eq:Inclusion-excess-loss-class-to-class}, one has, for any
  $r>0$ and $\lambda>0$,
  \begin{equation*}
    \E \norm{P - P_n}_{\cL_{r, \lambda}} \leq \E \E_\epsilon \frac 2n
    \sup_{A \in A_r^* + K_{r, \lambda} }  \Big|\sum_{i=1}^n \epsilon_i \cL_{r,
      A}(X_i, Y_i)\Big|,
  \end{equation*}
  where $\epsilon_1,\ldots,\epsilon_n$ are $n$ i.i.d Rademacher
  variables. Introduce the Rademacher process $Z_A :=
  \sum_{i=1}^n \epsilon_i \cL_{r, A}(X_i, Y_i)$, and note that for any
  $A, A' \in A_r^* + K_{r, \lambda}$:
  \begin{align*}
    \E_\epsilon |Z_A - Z_{A'}|^2 &= \sum_{i=1}^n \prodsca{X_i}{A - A'}^2 (2
    Y_i - \prodsca{X_i}{A + A'})^2 \\
    &= 4 \sum_{i=1}^n \prodsca{X_i}{A - A'}^2 (Y_i -
    \prodsca{X_i}{A_r^*} - \prodsca{X_i}{\frac{A + A'}{2} - A_r^*})^2 \\
    &\leq 8 \norm{A - A'}_{n, \infty}^2 \Big( \sum_{i=1}^n
    (Y_i-\inr{X_i,A_r^*})^2 + \sup_{A \in K_{r, \lambda}} \sum_{i=1}^n
    \prodsca{X_i}{A}^2 \Big),
  \end{align*}
  where we recall that $\norm{A}_{n, \infty} = \max_{i=1, \ldots, n}
  |\prodsca{X_i}{A}|$.  So, using the generic chaining mechanism
  (cf. Theorem~\eqref{theo:Generic-chaining-sous-gaussien}), we obtain
  \begin{align*}
    \E \norm{P - P_n}_{\cL_{r, \lambda}} &\leq \frac{c}{n} \E \Big[
    \gamma_2(K_{r, \lambda}, \norm{\cdot}_{n, \infty}) \Big(
    \sum_{i=1}^n (Y_i-\inr{X_i,A_r^*})^2 + \sup_{A \in K_{r, \lambda}}
    \sum_{i=1}^n \prodsca{X_i}{A}^2 \Big)^{1/2} \Big]  \\
    &\leq \frac{c}{\sqrt n} ( \E \gamma_2(K_{r, \lambda},
    \norm{\cdot}_{n, \infty})^2)^{1/2} \Big( R(A_r^*) + \E
    \sup_{A \in K_{r, \lambda}} \frac1n \sum_{i=1}^n
    \prodsca{X_i}{A}^2 \Big)^{1/2}.
  \end{align*}
  Now, introduce, for some set $K \subset \cM_{m, T}$ the functional
  \begin{equation*}
    U_n(K) := ( \E \gamma_2(K, \norm{\cdot}_{n, \infty})^2)^{1/2}.
  \end{equation*}
  Using Theorem~1.2 from \cite{MR2321621}, we obtain:
  \begin{equation*}
    \E \sup_{A \in K_{r, \lambda}} \frac1n \sum_{i=1}^n \prodsca{X_i}{A}^2
    \leq  \lambda + c \max\Big( \sqrt{\frac{\lambda}{n}} U_n(K_{r, \lambda}),
    \frac{U_n(K_{r, \lambda})}{n}^2 \Big),
  \end{equation*}
  and so, we arrive at
  \begin{equation*}
    \E \norm{P - P_n}_{\cL_{r, \lambda}} \leq c \phi_n(r, \lambda),
  \end{equation*}
  where
  \begin{align*}
    \phi_n(r, \lambda) &:= c\frac{U_n(K_{r, \lambda})}{\sqrt n} \Big(
    \lambda + R(A_r^*) + \frac{\sqrt \lambda U_n(K_{r,
        \lambda})}{\sqrt n} + \frac{U_n(K_{r, \lambda})^2}{n}
    \Big)^{1/2} \\
    &\leq c \Big( U_n(K_{r, \lambda}) \sqrt \frac{\lambda}{n} +
    U_n(K_{r, \lambda}) \sqrt \frac{ R(A_r^*)}{n} +
    \frac{U_n(K_{r, \lambda})^2}{n} \Big).
  \end{align*} 
  We conclude with the peeling argument provided in Lemma~4.6 of
  \cite{Mendelson08regularizationin}:
  \begin{equation*}
    \E\norm{P-P_n}_{V_{r,\lambda}} \leq c\sum_{i\geq0} 2^{-i}
    \E\norm{P-P_n}_{\cL_{r,2^{i+1}\lambda}}.     \qedhere
  \end{equation*}
\end{proof}

Now, we can derive the following corollary. It gives several upper
bounds for $\E\norm{P - P_n}_{V_{r,\lambda}}$, depending on what $B_r$
is (i.e. which penalty function is used).

\begin{corollary}[$\norm{\cdot}_{S_1}$ penalization]
  \label{coro:iso-function-S1}
  Let Assumptions~\ref{ass:X} and~\ref{ass:Y} hold and assume that
  $B_r = B_{r, 1, 0, 0}$ for $r > 0$, see~\eqref{eq:Br}.  Then, we
  have
  \begin{equation*}
    \E\norm{P - P_n}_{V_{r,\lambda_1(r)}} \leq \frac{\lambda_1(r)}{8}
  \end{equation*}
  for any $r > 0$, where
  \begin{equation*}
    \lambda_1(r) = c \Big( \frac{b_{X, 2}^2 r^2 (\log n)^2}{n}
    + \frac{b_{X, 2} b_Y r \log n }{\sqrt n} \Big).
  \end{equation*}
\end{corollary}

\begin{proof}
  If $B_r = r B(S_1)$, we have using the embedding $K_{r, \lambda}
  \subset 2 B_r$ and Proposition~\ref{prop:complexity} that $U_n(K_{r,
    \lambda}) \leq c b_{X, 2} r \log n$, so
  \begin{align*}
    \phi_n&(r, \lambda) \leq c \Big( b_{X, 2} r \log n \sqrt
    \frac{\lambda}{n} + b_{X, 2} r \log n \sqrt \frac{R(A_r^*)}{n} +
    \frac{b_{2, X}^2 r^2 (\log n)^2}{n} \Big) =: c \phi_{n, 1}(r, x).
  \end{align*}  
  Hence, using Proposition~\ref{prop:penalty} we obtain
  \begin{align*}
    \E\norm{P - P_n}_{V_{r,\lambda}} &\leq c \sum_{i \geq 0} 2^{-i}
    \phi_{n, 1}(r, 2^{i+1} \lambda) \leq c \phi_{n, 1}(r, \lambda),
  \end{align*}
  where we used the fact that the sum is comparable to its first term
  because of the exponential decay of the summands.  Thus, one has
  $\E\norm{P - P_n}_{V_{r,\lambda}} \leq \lambda / 8$ when $\lambda
  \geq c \phi_{n, 1}(r, \lambda)$. In particular, since $R(A_r^*) \leq
  \E Y^2 \leq b_Y^2$ (see Assumption~\ref{ass:Y}), for values of
  $\lambda$ such that
  \begin{equation*}
    \lambda \geq c \Big( \frac{b_{X, 2}^2 r^2 (\log n)^2}{n}
    + \frac{b_{X, 2} b_Y r \log n }{\sqrt n} \Big),
  \end{equation*} 
  we have $\E\norm{P - P_n}_{V_{r,\lambda}} \leq \lambda / 8$, which
  proves the Corollary.
\end{proof}  

\begin{corollary}[$\norm{\cdot}_{S_1} + \norm{\cdot}_1$ penalization]
  \label{coro:iso-function-S1L1}
  Let Assumptions~\ref{ass:X} and~\ref{ass:Y} hold and assume that
  $B_r = B_{r, r_1, 0, r_3}$ for $r, r_1, r_3 > 0$,
  see~\eqref{eq:Br}. Then, we have
  \begin{equation*}
    \E\norm{P - P_n}_{V_{r, \lambda_{r_1, 0, r_3}(r)}} \leq
    \frac{\lambda_{r_1, 0, r_3}(r)}{8}
  \end{equation*}
  for any $r > 0$, where
  \begin{equation*}
    \lambda_{r_1, 0, r_3}(r) = c \Big[ \Big( \frac{1}{r_1^2}
    \wedge \frac{\log(mT)}{r_3^2} \Big) \frac{b_{X, 2}^2 r^2 (\log
      n)^2}{n} + \Big( \frac{1}{r_1} \wedge
    \frac{\sqrt{\log(mT)}}{r_3} \Big) \frac{b_{X, 2} b_Y r
      (\log n)^{3/2}}{\sqrt n} ) \Big]. 
  \end{equation*}
\end{corollary}

\begin{proof}
  The proof follows the same steps as the proof of
  Corollary~\ref{coro:iso-function-S1}.
\end{proof}

\begin{corollary}[$\norm{\cdot}_{S_1} + \norm{\cdot}_{S_2}^2$
  penalization]
  \label{coro:iso-function-S1S2}
  Let Assumptions~\ref{ass:X} and~\ref{ass:Y} hold and assume that
  $B_r = B_{r, r_1, r_2, 0}$ for $r, r_1, r_2 > 0$,
  see~\eqref{eq:Br}. Then, we have
  \begin{equation*}
    \E\norm{P - P_n}_{V_{r,\lambda_{r_1, r_2}(r)}} \leq \frac{\lambda_{r_1,r_2}(r)}{8}
  \end{equation*}
  for any $r > 0$, where
  \begin{equation*}
    \lambda_{r_1, r_2}(r) = c \Big( \frac{b_{X, 2}^2 r (\log n)^2}{r_2 n}
    + \frac{b_{X, 2} b_Y r \log n}{r_1 \sqrt n} \Big).
  \end{equation*}
\end{corollary}

\begin{proof}
  Use the inclusion
  \begin{equation*}
    B_r \subset \sqrt{\frac{r}{r_2}} B(S_2) \cap \frac{r}{r_1}
    B(S_1)
  \end{equation*}
  to obtain using Proposition~\ref{prop:complexity} that
  \begin{align*}
    \phi_n&(r, \lambda) \leq c\Big( b_{X, 2} \sqrt{\frac{r}{r_2}} \log
    n \sqrt\frac{\lambda}{n} + b_{X, 2} \frac{r}{r_1} \log n \sqrt
    \frac{R(A_r^*)}{n} + \frac{b_{X, 2}^2 r (\log n)^2}{r_2 n} \Big).
  \end{align*}  
  The remaining of the proof is the same as the one of
  Corollary~\ref{coro:iso-function-S1} so it is omitted.
\end{proof}

\begin{corollary}[$\norm{\cdot}_{S_1} + \norm{\cdot}_{S_2}^2 + \norm{\cdot}_1$
  penalization]
  \label{coro:iso-function-S1S2L1}
  Let Assumptions~\ref{ass:X} and~\ref{ass:Y} hold and assume that
  $B_r = B_{r, r_1, r_2, r_3}$ for $r, r_1, r_2, r_3 > 0$,
  see~\eqref{eq:Br}. Then, we have
  \begin{equation*}
    \E\norm{P - P_n}_{V_{r, \lambda_{r_1, r_2, r_3}(r)}} \leq \frac{\lambda_{r_1,
        r_2, r_3}(r)}{8}
  \end{equation*}
  for any $r > 0$, where
  \begin{equation*}
    \lambda_{r_1, r_2, r_3}(r)  = c \Big[ \frac{b_{X, 2}^2 r (\log
      n)^2}{r_2 n}  + \Big( \frac{1}{r_1} \wedge
    \frac{\sqrt{\log(mT)}}{r_3} \Big) \frac{b_{X, 2} b_Y r
      (\log n)^{3/2}}{\sqrt n} ) \Big]. 
  \end{equation*}
\end{corollary}

\begin{proof}
  The proof follows the same steps as the proof of
  Corollary~\ref{coro:iso-function-S1S2}.
\end{proof}

The main difference between $\lambda_1(r), \lambda_{r, r_1, 0,
  r_3}(r)$ and $\lambda_{r_1, r_2}(r), \lambda_{r_1, r_2, r_3}(r)$ is
that $\lambda_{r_1, r_2}(r)$ and $\lambda_{r_1, r_2, r_3}(r)$ are
linear in $r$ while $\lambda_1(r)$ and $\lambda_{r_1, 0, r_3}(r)$ are
quadratic. The analysis of the isomorphic functions with quadratic
terms will require an extra argument in the proof, in order to remove
them from the penality (see below).

\begin{remark}[Localization does not work here]
  Note that, in Corollaries~\ref{coro:iso-function-S1}
  to~\ref{coro:iso-function-S1L1}, we don't use the fact that $K_{r,
    \lambda} \subset \sqrt \lambda D$, that is, we don't use the
  localization argument which usually allows to derive fast rates in
  statistical learning theory. Indeed, for the matrix completion
  problem, one has $\E \prodsca{X}{A - A_r^*}^2 = \frac{1}{m T}
  \norm{A - A_r^*}_{S_2}^2$, so when $\E \prodsca{X}{A - A_r^*}^2 \leq
  \lambda$, we only know that $A \in A_r^* + \sqrt{m T \lambda}
  B(S_2)$, leading to a term of order $m T / n$ (up to logarithms) in
  the isomorphic function. This term is way too large, since one has
  typically in matrix completion problems that $mT \gg n$.
\end{remark}

\subsection{Isomorphic penalization method}

We introduce the \emph{isomorphic penalization method} developed by
P. Bartlett, S. Mendelson and J. Neeman in the following general
setup. Let $(\cZ,\sigma_\cZ,\nu)$ be a measurable space endowed with
the probability measure $\nu$. We consider $Z,Z_1,Z_2,\ldots,Z_n$
i.i.d. random variables having $\nu$ for common probability
distribution. We are given a class $\cF$ of functions on a measurable
space $(\cX,\sigma_\cX)$, a loss function and a risk function
\begin{equation*}
  Q : \cZ \times \cF \rightarrow \mathbb{R} ; \quad R(f) = \E Q(Z,f).
\end{equation*}
For the problem we have in mind, we will use $Q((X,Y), A) = (Y -
\inr{X,A})^2$ for every $A \in \cM_{m, T}$.

Now, we go into the core of the isomorphic penalization method. We are
given a model $F\subset\cF$ and a family $\{F_r:r\geq0\}$ of subsets
of $F$. We consider the following definition.

\begin{definition}[cf.~\cite{Mendelson08regularizationin}]
  \label{Def:Ordered_Parametrized_Hierarchy}
  Let $\rho_n$ be a non-negative function defined on
  $\R_+\times\R_+^*$ (which may depend on the sample). We say that the
  family $\{F_r:r\geq0\}$ of subsets of $F$ is an ordered,
  parameterized hierarchy of $F$ with isomorphic function $\rho_n$
  when the following conditions are satisfied:
\begin{enumerate}
\item $\{F_r:r\geq0\}$ is non-decreasing (that is $s\leq t\Rightarrow
  F_s\subseteq F_t$);
\item for any $r\geq0$, there exists a unique element $f^*_r\in F_r$
  such that $R(f^*_r)=\inf(R(f):f\in F_r)$; we consider the excess
  loss function associated with the class $F_r$
  \begin{equation}\label{eq:ExcessLossFr}
    \cL_{r,f}(\cdot)=Q(\cdot,f)-Q(\cdot,f_r^*);
  \end{equation}
\item the map $r\longmapsto R(f^*_r)$ is continuous;
\item for every $r_0\geq0$, $\cap_{r\geq r_0}F_r=F_{r_0}$;
\item $\cup_{r\geq0}F_r=F$;
\item for every $r\geq0$ and $u>0$, with probability at least $1-\exp(-u)$
  \begin{equation}\label{eq:IsomorphicProperty}
     (1/2)P_n\cL_{r,f}-\rho_n(r,u)\leq P\cL_{r,f}\leq 2P_n\cL_{r,f}+\rho_n(r,u),
  \end{equation}for any $f\in F_r$ and  $P_n\cL_{r,f}=(1/n)\sum_{i=1}^n\cL_{r,f}(Z_i)$.
\end{enumerate}
\end{definition}

In the context of learning theory, ordered, parametrized hierarchy of
a set $F$ with isomorphic function $\rho_n$ provides a very general
framework for the construction of penalized empirical risk
minimization procedure. The following result from
\cite{Mendelson08regularizationin} proves that the isomorphic function
is a ``correct penalty function''.

\begin{theorem}[\cite{Mendelson08regularizationin}]
  \label{Theo:MN:08}
  There exists absolute positive constants $c_1$ and $c_2$ such that
  the following holds.  Let $\{F_r:r\geq0\}$ be an ordered,
  parameterized hierarchy of $F$ with isomorphic function
  $\rho_n$. Let $u>0$. With probability at least $1-\exp(-u)$ any
  penalized empirical risk minimization procedure
  \begin{equation}\label{eq:PERM1}
    \hat{f} \in \argmin_{f\in F} \Big(R_n(f) + c_1 \rho_n(2(r(f)
    + 1), \theta(r(f)+1,u)) \Big),
  \end{equation}
  where $r(f)=\inf(r\geq0: f\in F_r)$ and
  $R_n(f)=(1/n)\sum_{i=1}^nQ(Z_i,f)$ is the empirical risk of $f$,
  satisfies
  \begin{equation*}
    R(\hat{f})\leq \inf_{f\in F}\Big( R(f) + c_2
    \rho_n(2(r(f)+1),\theta(r(f)+1,u)) \Big)
  \end{equation*}
  where for all $r\geq1$ and $x>0$,
  \begin{equation*}
    \theta(r, x) = x + \ln(\pi^2/6) + 2 \ln\Big(1 +
    \frac{R(f_0^*)}{\rho_n(0,x+\log(\pi^2/6))}+\log r\Big).
  \end{equation*}
\end{theorem}

\subsection{End of the proof of Theorems~\ref{thm:oracle-S1}
  and~\ref{thm:oracle-S1-S2}}
\label{sec:Proofs-of-the-two-main-results}

First, we need to prove that the family of models $\{B_r : r \geq 0
\}$ is an ordered, parametrized hierarchy of $\cM_{m,T}$. First,
fourth and fifth points of
Definition~\ref{Def:Ordered_Parametrized_Hierarchy} are easy to
check. Second point follows from Lemma~\ref{lem:convexity}. For the
third point, we consider $0\leq q<r<s$, $\beta:=q/r$ and
$\alpha:=r/s$. Since $\alpha A_s^*\in B_r$, we have
\begin{equation*}
  0 \leq R(A_r^*) - R(A_s^*) \leq R(\alpha
  A_s^*) - R(A_s^*) \leq (\alpha^2 - 1) \norm{\inr{X, A_s^*}}_{L^2}^2
  + 2(1-\alpha) \norm{Y}_2 \norm{\inr{X,A_s^*}}_{L^2}.
\end{equation*}
As $s\rightarrow r$, the rights hand side tends to zero (because
$\inr{X,A_s^*}$ are uniformly bounded in $L_2$ for $s\in[r,r+1]$). So
$r \mapsto R(A_r^*)$ is upper semi-continuous on $(0,\infty)$. The
continuity in $r=0$ follows the same line. In the other direction,
\begin{equation*}
  0 \leq R(A_q^*) - R(A_r^*) \leq R(\beta
  A_r^*) - R(A_r^*) \leq (\alpha^2 - 1) \norm{\inr{X,A_r^*}}_{L^2}^2 +
  2(1-\alpha) \norm{Y}_2 \norm{\inr{X, A_r^*}}_{L^2}
\end{equation*}
and the right hand side tends to zero for the same reason as before.

Now, we turn to the sixth point of
Definition~\ref{Def:Ordered_Parametrized_Hierarchy}. That is the
computation of the isomorphic function $\rho_n$ associated with the
family $\{B_r:r\geq0\}$. Using Theorem~\ref{cor:isomorphy} we obtain
that, with a probability larger than $1 - 4 e^{-x}$:
\begin{equation*}
  \frac 12 P_n \cL_{r, A} - \rho_n(r, x) \leq P \cL_{r, A} \leq 2
  P_n \cL_{r, A} + \rho_n(r, x) \quad \forall A \in B_r,
\end{equation*}
where
\begin{equation*}
  \rho_n(r, x) := c \Big[ \lambda(r) + \big( b_Y' + C_r\big)^2
  \Big(\frac{x\log n}{n}\Big)\Big],
\end{equation*}
where $b_Y' := b_{Y,\psi_1} + b_{Y,\infty} + b_{Y,2}$, where $C_r$ and
$\lambda(r)$ are defined depending on the considered penalization
(see~\eqref{eq:parameter-Cr} and
Corollaries~\ref{coro:iso-function-S1}
to~\ref{coro:iso-function-S1S2L1}). Now, we apply
Theorem~\ref{Theo:MN:08} to the hierarchy $F_r = B_r$ for $r \geq 0$.
First of all, note that, for every $x > 0$ and $r \geq 1$
\begin{align*}
  \theta(r, x) &= x + \ln(\pi^2/6) + 2 \ln\Big(1 + \frac{\E
    Y^2}{\rho_n(0, x + \log(\pi^2/6))} + \log r\Big) \\
  &\leq x + c( \log n + \log \log r),
\end{align*}
so $\rho_n(2(r+1), \theta(r+1, x)) \leq \rho_n'(r, x)$, with:
\begin{equation*}
  \rho_n'(r, x) := c \Big[ \lambda (2(r + 1)) +
  (b_Y' + C_r )^2 \frac{(x + \log n + \log\log r)\log n}{n}\Big].
\end{equation*}
From now on, the analysis depends on the penalization, so we consider
them separately.

\subsubsection{The $\norm{\cdot}_{S_1}$ case}

Recall that in this case
\begin{equation*}
  \lambda(r) = c \Big( \frac{b_{X, 2}^2 r^2 (\log
    n)^2}{n} + \frac{b_{X, 2} b_Y r \log n }{\sqrt n} \Big) 
\end{equation*}
and $C_r = b_{X, \infty} r$, see~\eqref{eq:parameter-Cr}. An easy
computation gives $\rho_n'(r, x) \leq \tilde \rho_{n, 1}(r, x)$ where
\begin{align*}
  \tilde \rho_{n, 1}(r, x) := c_{X, Y} \frac{ (r+1)^2 (x + \log n \vee
    \log \log r) \log n}{n} \vee p_{n, 1}(r, x),
\end{align*}
where $c_{X, Y} := c (1 + b_{X, 2}^2 + b_Y b_X + b_{Y, \psi_1}^2 +
b_{Y, \infty}^2 + b_{Y, 2}^2 + b_{X, \infty}^2)$ and where
\begin{equation*}
  p_{n, 1}(r, x) :=  c_{X, Y} \frac{ (r + 1) (x + \log n)  \log n}{\sqrt n}.
\end{equation*}
Note that $p_{n, 1}(r, x)$ is the penalty we want (the one considered
in Theorem~\ref{thm:oracle-S1}). Let us introduce for short $r(A) =
\norm{A}_{S_1}$ and the following functionals:
\begin{align*}
  \Lambda_1(A) &= R(A) + \pen_1(A), \quad \Lambda_{n, 1}(A) = R_n(A) +
  \pen_1(A), \\
  \tilde \Lambda_1(A) &= R(A) + \tilde \pen_1(A), \quad \tilde
  \Lambda_{n, 1}(A) = R_n(A) + \tilde \pen_1(A),
\end{align*}
where $\pen_1(A) := p_{n, 1}(r(A), x)$ and where $\tilde \pen_1(A) :=
\tilde \rho_{n, 1}(r(A), x)$ is a penalization that satisfies that, if
$\tilde A \in \argmin_{A} \tilde \Lambda_{n, 1}(A)$, then we have
$R(\tilde A) \leq \inf_{A} \tilde \Lambda_1(A)$ with a probability
larger than $1 - 4e^{-x}$. Recall that we want to prove that if $\hat
A \in \argmin_{A} \Lambda_{n, 1}(A)$, then we have $R(\hat A) \leq
\inf_{A} \Lambda_1(A)$ with a probability larger than $1 - 5
e^{-x}$. This will follow if we prove
\begin{align}
  \label{eq:end-argument1}
  &\inf_A \tilde \Lambda_1(A) \leq \inf_A \Lambda_1(A) \quad \text{ and }
  \\
  \label{eq:end-argument2}
  &\argmin_{A} \Lambda_{n, 1}(A) \subset \argmin_{A} \tilde
  \Lambda_{n, 1}(A),
\end{align}
so we focus on the proof of these two facts. First of all, let us
prove that if $\tilde \rho_{n, 1}(r, x) > p_{n, 1}(r, x)$ then both
$r$ and $p_{n, 1}(r, x)$ cannot be small.

If $\log n < \log \log r$ we have $r > e^n$ and $p_{n, 1}(x, r) >
c_{X, Y} e^n (\log n)^2 / \sqrt n$. If $\log n \geq \log \log r$ and
$\tilde \rho_{n, 1}(r, x) > p_{n, 1}(r, x)$, then
\begin{align*}
  \frac{(r+1)^2 (x + \log n) \log n }{n} > \frac{(r + 1) (x + \log n)
    \log n}{\sqrt n},
\end{align*}
so $r > \sqrt n - 1$ and $p_{n, 1}(r, x) > c_{X, Y} (\log n)^2$.
Hence, we proved that if $\tilde \rho_{n, 1}(r, x) > p_{n, 1}(r, x)$,
then $r > 1$ and $p_{n, 1}(r, x) > c_{X, Y} (\log n)^2$. Note also
that $p_{n, 1}(r, x) > 2 (x + \log n) \log n / \sqrt n$ since $r > 1$.

Let us turn to the proof of~\eqref{eq:end-argument1}. Let $A'$ be such
that $\tilde \Lambda_1(A') > \Lambda_1(A')$. Then $\tilde \pen_1(A') >
\pen_1(A')$, ie $\tilde \rho_{n, 1}(r(A'), x) > p_{n, 1}(r(A'), x)$,
so that $r(A') > 1$, $p_{n, 1}(r(A'), x) > c_{X, Y} (\log n)^2$ and
$p_{n, 1}(r(A'), x) > 2 c_{X, Y} (x + \log n) \log n / \sqrt n$. On
the other hand, we have $\inf_A \Lambda_1(A) \leq b_Y^2 + \pen_1(0) =
b_Y^2 + p_{n, 1}(0, x)$. But $p_{n, 1}(r(A'), x) > c_{X, Y} (\log n)^2
> 2 b_Y^2$ and $p_{n, 1}(r(A'), x) > 2 p_{n, 1}(0, x)$ since $r(A') >
1$, so that $b_Y^2 + p_{n, 1}(0, x) < p_{n, 1}(r(A'), x)$ and then
\begin{equation*}
  \inf_A \Lambda_1(A) < p_n(r(A'), x) \leq \Lambda_1(A').
\end{equation*}
Hence, we proved that if $A'$ is such that $\Lambda_1(A') \leq \inf_A
\Lambda_1(A)$, we have $\tilde \Lambda_1(A') \leq \Lambda_1(A')$, so
$\inf_A \tilde \Lambda_1(A) \leq \tilde \Lambda_1(A') \leq
\Lambda_1(A') \leq \inf_A \Lambda_1(A)$, which
proves~\eqref{eq:end-argument1}.

The proof of~\eqref{eq:end-argument2} is almost the same. Let $A'$ be
such that $\tilde \Lambda_{n, 1}(A') > \Lambda_{n, 1}(A')$, so as
before we have $r(A') > 1$, $p_{n, 1}(r(A'), x) > c_{X, Y} (\log n)^2$
and $p_{n, 1}(r(A'), x) > 2 c_{X, Y} (x + \log n) \log n / \sqrt
n$. This time we have $\inf_A \Lambda_{n, 1}(A) \leq n^{-1}
\sum_{i=1}^n Y_i^2 + p_{n, 1}(0, x)$, so we use some concentration for
the sum of the $Y_i^2$'s. Indeed, we have, as a consequence of
\cite{MR2424985}, that
\begin{equation}
  \label{eq:adamcazk-one}
  \frac{1}{n} \sum_{i=1}^n Y_i^2 \leq \E Y^2 + c_1 \sqrt{\E(Y^4) \frac
    xn} + c_2 \log n\frac{\norm{Y^2}_{\psi_1} x}{n}
\end{equation}
with a probability larger than $1 - e^{-x}$. But then, it is easy to
infer that for $n$ large enough, the right hand side of
(\ref{eq:adamcazk-one}) is smaller than $p_{n, 1}(r(A'), x) / 2$, so
that we have, on an event of probability larger than $1 - e^{-x}$,
that
\begin{equation*}
  \inf_A \Lambda_{n, 1}(A) \leq \frac 1n \sum_{i=1}^n Y_i^2 + p_{n,
    1}(0, x) < p_{n, 1}(r(A'), x) < \Lambda_{n, 1}(A').
\end{equation*}
So, we proved that if $\Lambda_{n, 1}(A') < \tilde \Lambda_{n,
  1}(A')$, then $A' \notin \argmin_A \Lambda_{n, 1}(A)$, or
equivalently that $\argmin_A \Lambda_{n, 1}(A) \subset \{ A : \tilde
\Lambda_{n, 1}(A) \leq \Lambda_{n, 1}(A) \}$. But $\Lambda_{n, 1}(A)
\leq \tilde \Lambda_{n, 1}(A)$ for any $A$ (since $p_{n, 1}(r, x) \leq
\tilde \rho_{n, 1}(r, x)$), so~\eqref{eq:end-argument2} follows. This
concludes the proof of Theorem~\ref{thm:oracle-S1}.

\subsubsection{The $\norm{\cdot}_{S_1} + \norm{\cdot}_1$ case}

Recall that in this case
\begin{equation*}
  \lambda(r) = c \Big[ \Big( \frac{1}{r_1}
  \wedge \frac{\sqrt{\log(mT)}}{r_3} \Big)^2 \frac{b_{X, 2}^2 r^2 (\log
    n)^2}{n} + \Big( \frac{1}{r_1} \wedge
  \frac{\sqrt{\log(mT)}}{r_3} \Big) \frac{b_{X, 2} b_Y r
    (\log n)^{3/2}}{\sqrt n} ) \Big],
\end{equation*}
and that
\begin{equation*}
  C_r = \min\Big(b_{X,\infty}\frac{r}{r_1}, b_{X, \ell_\infty}
  \frac{r}{r_3} \Big),
\end{equation*}
see~\eqref{eq:parameter-Cr}. An easy computation gives that
$\rho_n'(r, x) \leq \tilde \rho_{n, 2}(r, x)$, where
\begin{align*}
  \tilde \rho_{n, 2}(r, x) := c_{X, Y} \Big( \frac{1}{r_1} \wedge
  \frac{\sqrt{\log(m T)}}{r_3}\Big)^2 \frac{(r+1)^2(x + \log n \vee
    \log \log r) \log n}{n} \vee p_{n, 2}(r, x),
\end{align*}
where $c_{X, Y} = c (1 + b_{X, 2}^2 + b_{X, 2} b_Y + b_{Y, \psi_1}^2 +
b_{Y, \infty}^2 + b_{Y, 2}^2 + b_{X, \infty}^2 + b_{X,
  \ell_\infty}^2)$ and
\begin{equation*}
  p_{n, 2}(r, x) := c_{X, Y} \Big( \frac{1}{r_1} \wedge
  \frac{\sqrt{\log(m T)}}{r_3}\Big) \frac{(r+1)(x + \log n) (\log n)^{3/2}}{\sqrt n}. 
\end{equation*}
Note that $p_{n, 2}(r, x)$ is the penalization we want (the one
considered in Theorem~\ref{thm:oracle-S1L1}). Introducing $r(A) = r_1
\norm{A}_{S_1} + r_3 \norm{A}_1$, the remaining of the proof follows
the lines of the pure $\norm{\cdot}_{S_1}$ case, so it is omitted.

\subsubsection{The $\norm{\cdot}_{S_1} + \norm{\cdot}_{S_2}^2$ case}

This is easier than what we did for the $\norm{\cdot}_{S_1}$ case,
since we only have a $\log \log r$ term to remove from the
penalization. Recall that
\begin{equation*}
  \lambda(r) = c \Big( \frac{b_{X, 2}^2 r (\log
    n)^2}{r_2 n} + \frac{b_{X, 2} b_Y r \log n}{r_1 \sqrt n} \Big),
\end{equation*}
and
\begin{equation*}
  C_r = \min\Big( b_{X, \infty} \frac{r}{r_1}, b_{X,2}
  \sqrt{\frac{r}{r_2}} \Big) \leq b_{X,2}\sqrt{\frac{r}{r_2}},
\end{equation*}
so that $\rho_n'(r, x) \leq \tilde \rho_{n, 3}(r, x)$ where
\begin{align*}
  \tilde \rho_{n, 3}(r, x) = c_{X, Y} \frac{(r+1) \log n}{\sqrt n}
  \Big( \frac{1}{r_1} + \frac{(x + \log n \vee \log\log r) \log n}{r_2
    \sqrt n} \Big),
\end{align*}
where $c_{X, Y} = c (1 + b_{X, 2}^2 + b_{X, 2} b_Y + b_{Y, \psi_1}^2 +
b_{Y, \infty}^2 + b_{Y, 2}^2)$. This is almost the penalty we want, up
to the $\log \log r$ term, so we consider.
\begin{equation*}
  p_{n, 3}(r, x) = c_{X, Y} \frac{(r+1) \log n}{\sqrt
    n} \Big( \frac{1}{r_1} + \frac{(x + \log n) \log
    n}{r_2 \sqrt n} \Big),
\end{equation*}
Let us introduce for short
\begin{equation*}
  r(A) := r_1 \norm{A}_{S_1} + r_2 \norm{A}_{S_2}^2 = \inf \big(r \geq
  0 : A \in B_r \big)
\end{equation*}
and the following functionals:
\begin{align*}
  \Lambda_3(A) &= R(A) + \pen_3(A), \quad \Lambda_{n, 3}(A) = R_n(A) +
  \pen_3(A), \\
  \tilde \Lambda_3(A) &= R(A) + \tilde \pen_3(A), \quad \tilde
  \Lambda_{n, 3}(A) = R_n(A) + \tilde \pen_3(A),
\end{align*}
where $\pen_3(A) := p_{n, 3}(r(A), x)$ and where $\tilde \pen_3(A) :=
\tilde \rho_{n, 3}(r(A), x)$. We only need to prove that
\begin{align}
  \label{eq:end-argument1bis}
  &\inf_A \tilde \Lambda_3(A) \leq \inf_A \Lambda_3(A) \quad \text{ and }
  \\
  \label{eq:end-argument2bis}
  &\argmin_{A} \Lambda_{n, 3}(A) \subset \argmin_{A} \tilde
  \Lambda_{n, 3}(A).
\end{align}
Obviously, if $\tilde \rho_{n, 3}(r, x) > p_{n, 3}(r, x)$, then $r >
e^n$, so following the arguments we used for the $S_1$ penalty, it is
easy to prove both~\eqref{eq:end-argument1bis}
and~\eqref{eq:end-argument2bis}. This concludes the proof of
Theorem~\ref{thm:oracle-S1-S2}.

\subsubsection{The $\norm{\cdot}_{S_1} + \norm{\cdot}_{S_2}^2 +
  \norm{\cdot}_1$ case}

Recall that in this case
  \begin{equation*}
    \lambda(r)  = c \Big[ \frac{b_{X, 2}^2 r (\log
      n)^2}{r_2 n}  + \Big( \frac{1}{r_1} \wedge
    \frac{\sqrt{\log(mT)}}{r_3} \Big) \frac{b_{X, 2} b_Y r
      (\log n)^{3/2}}{\sqrt n} \Big].
  \end{equation*}
  and that
\begin{equation}
  C_r = \min\Big(b_{X,\infty}\frac{r}{r_1},
  b_{X,2}\sqrt{\frac{r}{r_2}}, b_{X, \ell_\infty} \frac{r}{r_3} \Big)
  \leq b_{X,2}\sqrt{\frac{r}{r_2}},
\end{equation}
see~\eqref{eq:parameter-Cr}. An easy computation gives that
$\rho_n'(r, x) \leq \tilde \rho_{n, 4}(r, x)$, where
\begin{align*}
  \tilde \rho_{n, 4}(r, x) := c_{X, Y} \frac{(r+1) (\log n)^{3/2}}{\sqrt n}
  \Big( \frac{1}{r_1} \wedge \frac{\sqrt{\log(m T)}}{r_3} + \frac{x +
    \log n \vee \log \log r}{r_2 \sqrt n} \Big)
\end{align*}
where $c_{X, Y} = c (1 + b_{X, 2}^2 + b_{X, 2} b_Y + b_{Y, \psi_1}^2 +
b_{Y, \infty}^2 + b_{Y, 2}^2)$. The penalization we want is
\begin{align*}
  p_{n, 4}(r, x) := c_{X, Y} \frac{(r+1) (\log n)^{3/2}}{\sqrt n} \Big(
  \frac{1}{r_1} \wedge \frac{\sqrt{\log(m T)}}{r_3} + \frac{x + \log
    n}{r_2 \sqrt n} \Big),
\end{align*}
so introducing $r(A) = r_1 \norm{A}_{S_1} + r_2 \norm{A}_{S_2}^2 + r_3
\norm{A}_1$ and following the lines of the proof of the $S_1 + S_2$
case to remove the $\log \log r$ term, it is easy to conclude the
proof of Theorem~\ref{thm:oracle-S1S2L1}.

\bibliographystyle{plain}

\bibliography{biblio}

\end{document}